\documentclass[11pt,reqno]{amsart}
\usepackage{comment} 
\usepackage{amsmath,amsthm,amssymb}
\usepackage{nccmath}
\usepackage{a4wide}
\usepackage[numbers]{natbib}
\usepackage[utf8]{inputenc}
\usepackage{latexsym}
\usepackage{amssymb}
\usepackage[mathscr]{eucal}
\usepackage{epsfig}
\usepackage{lscape}
\usepackage{amsfonts}
\usepackage{graphicx}
\usepackage[dvipsnames]{xcolor}
\usepackage{yhmath}
\usepackage{mathdots}
\usepackage{MnSymbol}
\usepackage{epsfig}
\usepackage{lineno}
\usepackage{tikz}
\usepackage{titlesec}
\usepackage{geometry}
\usepackage{cite}
\usepackage[absolute]{textpos}
\usepackage{hyperref}

\titleformat{\section}
    {\normalfont\Large\bfseries}  
    {\thesection}                  
    {1em}                          
    {}                             

\titleformat{\subsection}
    {\large\bfseries}  
    {\thesubsection}                  
    {1em}                          
    {}                             

\newcommand{\KF}[5]{F^{#1}_{#2}\left[{#3\atop #4}\Bigg\vert #5\right]}

\newcommand{\NN}{\mathbb{N}}
\newcommand{\ZZ}{\mathbb{Z}}
\newcommand{\RR}{\mathbb{R}}

\newcommand{\dr}{\mathrm{d}}

\usepackage[colorinlistoftodos,prependcaption,textsize=small]{todonotes}

\definecolor{greenpie}{rgb}{0.69, 0.95, 0.76}

\newtheorem{theorem}{Theorem}[section]
\newtheorem{lemma}[theorem]{Lemma}
\newtheorem{corollary}[theorem]{Corollary}

\newtheorem{proposition}[theorem]{Proposition}

\theoremstyle{definition}
\newtheorem{definition}[theorem]{Definition}
\newtheorem{remark}[theorem]{Remark}
\usepackage{mathtools}  
\usepackage{mathrsfs}   

\def\ps@pprintTitle{%
     \let\@oddhead\@empty
     \let\@evenhead\@empty
     \let\@evenfoot\@oddfoot}

\title{A Rodrigues Formula for Multiple Orthogonal Polynomials on the Simplex}
\author{Lidia Fern\'andez, Ana Foulqui\'e-Moreno and Juan Antonio Villegas}

\address{Lidia Fern\'andez\\
IMAG and Departamento de Matem\'atica Aplicada\\
Universidad de Granada\\
Granada, Spain.}
\email{lidiafr@ugr.es}
\address{Ana Foulqui\'e-Moreno\\
Departamento de Matemática\\
Universidade de Aveiro\\
Aveiro,\linebreak Portugal}
\email{foulquie@ua.pt}
\address{Juan Antonio Villegas\\
IMAG and Departamento de Matem\'atica Aplicada\\
Universidad de Granada\\
Granada, Spain}
\email{jantoniovr@ugr.es}

\thanks{Corresponding author: Lidia Fernández  (\texttt{lidiafr@ugr.es}).}

\date{\today}



\begin{document}

\maketitle

\begin{abstract}

Rodrigues formulas play a central role in the construction of orthogonal polynomials associated with classical measures. In this paper, we introduce a family of multivariate multiple orthogonal polynomials on the simplex obtained through a Rodrigues-type construction that combines features of bivariate Jacobi polynomials on the simplex and univariate Jacobi--Piñeiro polynomials. We prove that the proposed construction produces polynomials and establish their main structural properties, including symmetry and multiple orthogonality with respect to several measures. This yields a natural multivariate extension of the classical Jacobi--Piñeiro family and, to the best of our knowledge, provides one of the first Rodrigues-type constructions for multivariate multiple orthogonal polynomials. Furthermore, we formulate a bivariate Hermite--Padé-type approximation problem and show that the proposed polynomial family naturally appears as a common denominator of the corresponding approximants. Numerical experiments illustrating the performance of the resulting approximations are also presented.

\end{abstract}


\section{Introduction}\label{sec:intro}

Within the framework of the theory of orthogonal polynomials \citep[Chapter 1]{Chi78}, the properties of the so-called classical weight functions are well known. These weights are functions $\omega(x)$ that satisfy the Pearson equation
\begin{align}
\label{eq:Pearson}
(\phi(x)\omega(x))^{\prime} + \varphi(x)\omega(x) = 0,
\end{align}
where $\phi(x)$ and $\varphi(x)$ are polynomials with $\deg(\phi)\le 2$ and $\deg(\varphi)=1$. Among other characteristic properties, orthogonal polynomials with respect to classical weights always admit a representation in terms of a Rodrigues formula. Indeed, if $\omega(x)$ is a classical weight satisfying \eqref{eq:Pearson}, then the associated sequence of  orthogonal polynomials $\{P_n\}$ can be expressed (up to a multiplicative constant) as
\begin{equation}
    \label{eq:Rodrigues-Jacobi}
    P_n(x)=\dfrac{1}{\omega(x)}\dfrac{d^n}{dx^n}(\omega(x)\phi(x)^n).
\end{equation}
Rodrigues formulas provide an explicit representation of the corresponding orthogonal polynomials and play a fundamental role in the study of their structural properties.

In recent years, one of the most active research topics in the theory of  orthogonal polynomials and special functions has been multiple orthogonality \citep{Ism05, VA20}. In this extension of the standard theory, polynomials, known as multiple orthogonal polynomials, satisfy orthogonality conditions with respect to more than one measure (see Section~\ref{sec:preliminaries}). Within this framework, $r$ classical weight functions $\omega_j$ are considered in \citep{ABVA03, BCVA05}. Moreover, explicit expressions for several families of type II multiple orthogonal polynomials with respect to classical measures are obtained by applying several Rodrigues operators. Analogous results for classical discrete measures are presented in \citep{ACVA01}.

The concept of the Rodrigues formula has also been generalized to the multivariate case using different approaches \citep{AFPP09,AK26, DX14, Herm1865, Koo75,Sue99}.

One of the most studied and well-known extensions of Jacobi polynomials to the bivariate setting are the families of orthogonal polynomials associated with the weight function
\begin{equation}
    \label{eq:jacobi-weight}
    W^{(\alpha,\beta,\gamma)}(x,y)=x^\alpha y^\beta (1-x-y)^\gamma,\qquad \text{with }\alpha,\beta,\gamma>-1,
\end{equation}
supported on the triangular domain
\begin{equation}
\label{eq:triangle}
    T = \{(x,y)\in\RR^2: x\geq 0,\; y\geq 0,\; x+y \leq 1\},
\end{equation}
see \citep[Section XXXIII]{AK26} or \citep[Section 2.4]{DX14}. In contrast to the univariate setting, the space
$\mathcal V_n^2(W^{(\alpha,\beta,\gamma)})$
has dimension $n+1$, and therefore orthogonal polynomials are not uniquely determined up to a multiplicative constant. As a consequence, several bases of $\mathcal V_n^2(W^{(\alpha,\beta,\gamma)})$ exist, where $\mathcal V_n^2(W^{(\alpha,\beta,\gamma)})$ denotes the space of orthogonal polynomials of total degree $n$ associated with the weight \eqref{eq:jacobi-weight}. For $0\leq k \leq n$, one such basis is given by the Rodrigues-type formula (see \citep[Proposition 2.4.3]{DX14}, \citep[Chapter V, Section 4]{Sue99} or \citep[Section XXXIII, Eq. (5)]{AK26}).

\begin{equation}
    \label{eq:Rodrigues-2-var}
    U^{(\alpha,\beta,\gamma)}_{k,n}(x,y) = \dfrac{1}{W^{(\alpha,\beta,\gamma)}(x,y)}\dfrac{\partial^n}{\partial x^{\,n-k}\partial y^{\,k}}\left(x^{\alpha+n-k}y^{\beta+k}(1-x-y)^{\gamma+n}\right).
\end{equation}
Observe that this expression constitutes a bivariate extension of the Rodrigues formula for Jacobi polynomials \citep[Section 4.3]{Szego75}
\begin{equation*}
    P_n^{(\alpha, \beta)}(x) = \frac{(-1)^n}{2^n n!} (1-x)^{-\alpha} (1+x)^{-\beta} \frac{d^n}{dx^n} \left[ (1-x)^{n+\alpha} (1+x)^{n+\beta} \right].
\end{equation*}

Multiple orthogonality on the real line and multivariate standard orthogonality have been, and still are, extensively researched. However, the extension of multiple orthogonality to other domains is taking its first steps. For example, in \citep{CDO15, KV24, KV26}, multiple orthogonality on the unit circle is studied. 

Regarding multi-dimensional domains, in \citep{FV26}, a general definition of type I and type II bivariate multiple orthogonal polynomials was introduced on the two-dimensional lattice, together with several extended results and examples. More recently, mixed-type bivariate multiple orthogonal polynomials on the step-line were studied in  \citep{MRW25}. To the best of our knowledge, no Rodrigues-type construction has been reported for multivariate multiple orthogonal polynomials. This paper aims to fill this gap in the particular framework of Jacobi-type weights on the simplex.

More precisely, we introduce a family of multiple orthogonal polynomials on the simplex defined through a Rodrigues-type construction that combines features of classical multivariate Jacobi polynomials and univariate Jacobi--Piñeiro polynomials. We prove that the proposed Rodrigues-type construction indeed produces polynomials and establish their main structural properties, including symmetry and multiple orthogonality with respect to several measures. This provides a natural extension of Jacobi--Piñeiro polynomials to the multivariate setting. Furthermore, motivated by the classical connection between multiple orthogonality and Hermite--Padé approximation, we formulate a bivariate Hermite--Padé-type approximation problem and show that the proposed polynomial family naturally arises as a common denominator of the corresponding approximants.

Another important aspect is the role played by orthogonal polynomials (and their extensions) in a variety of related fields \citep{MFVA16}. One of the most representative applications of (multiple) orthogonal polynomials is rational approximation of functions \citep{VA06}, which has also been studied in the multivariate setting; see, for example, \citep{CB00,Cuy83,Cuy86,Cuy99,CLY16}, where several extensions of Padé approximation are proposed. This connection becomes particularly relevant in the framework of multiple orthogonality, where Hermite--Padé approximation constitutes one of the main motivations for, and applications of, multiple orthogonal polynomials. In \citep{Sor02}, a new approach to Hermite--Padé approximation was introduced, and it serves as inspiration for the methodology developed in Section~\ref{sec:HP}.

The structure of the manuscript is as follows. In Section~\ref{sec:preliminaries}, we review the basic concepts of multiple orthogonality and their connection with Hermite--Padé approximation in the univariate setting. Section~\ref{sec:JP} introduces and studies a family of bivariate Jacobi--Piñeiro polynomials, obtained through a Rodrigues-type construction and characterized by multiple orthogonality conditions with respect to several Jacobi measures on the triangle. This section also includes extensions to an arbitrary number of measures and variables. Next, Section~\ref{sec:HP} formulates a bivariate Hermite--Padé-type approximation problem and shows how the proposed polynomial family appears naturally in its solution. In Section~\ref{sec:simulations}, we present numerical experiments illustrating the performance of the resulting approximants. Finally, Section~\ref{sec:conclusions} summarizes the main findings of the paper and discusses several directions for future research.

\section{Preliminaries: Multiple Orthogonal Polynomials}\label{sec:preliminaries}

As previously mentioned, Multiple Orthogonal Polynomials (MOP) extend the classical theory of orthogonal polynomials by satisfying orthogonality conditions with respect to several measures. In this section, we present the main definitions of multiple orthogonality in the univariate setting, together with the relevant notation and illustrative examples, with special emphasis on Jacobi--Piñeiro polynomials.

From now on, let us consider a number $r\geq 2$ of positive measures supported on subsets of $\RR$, denoted by $\mu_1,\dots,\mu_r$. Throughout this section, we assume that these measures are absolutely continuous, that is, there exist positive weight functions $\omega_j(x)$ such that $\dr\mu_j(x)=\omega_j(x)\dr x$, for $j=1,\dots,r$. This collection of measures is commonly referred to as a \emph{system of measures}, and it will be fixed for the remainder of the section.

There are two types of multiple orthogonality: type~I and type~II. In type~II multiple orthogonality, a single polynomial is considered, and its orthogonality conditions are distributed among the measures. In contrast, type~I multiple orthogonality involves $r$ distinct polynomials, possibly of different degrees, which jointly satisfy a system of orthogonality conditions. The way in which the conditions are split in type~II, as well as the degree constraints in type~I, is determined by a multi-index $\vec n=(n_1,\dots,n_r)\in\NN_0^r$, whose $\ell_1$ norm or modulus is defined as $|\vec n|=n_1+\cdots+n_r$.

More precisely, we define type~I and type~II MOP as follows:

\begin{definition}
    Given $\vec n=(n_1,\dots,n_r)\in\NN_0^r$, the type~I multiple orthogonal polynomials are $r$ polynomials $A_{\vec n,1}(x),\dots,A_{\vec n,r}(x)$ satisfying $\deg(A_{\vec n,j})\leq n_j-1$ for $j=1,\dots,r$, and
    \begin{equation}
        \sum_{j=1}^r \int_\RR A_{\vec n,j}(x)x^l\omega_j(x)\dr x = \begin{cases}
            0\qquad \text{ if }0\leq l <|\vec n|-1,\\
            1\qquad \text{ if }l =|\vec n|-1.
        \end{cases}
    \end{equation}
\end{definition}

\begin{definition}
    Given $\vec n=(n_1,\dots,n_r)\in\NN_0^r$, the type~II multiple orthogonal polynomial $P_{\vec n}(x)$ is the unique monic polynomial of degree $|\vec n|$ satisfying
    \begin{equation}
        \label{eq:orthogonality-1-var}
        \int_\RR P_{\vec n}(x)x^l\omega_j(x)\dr x = 0\qquad \text{for }0\leq l<n_j,\quad j=1,\dots,r.
    \end{equation}
\end{definition}

In this work, we focus exclusively on type~II MOP. For a more detailed introduction and a comprehensive discussion of their properties as extensions of standard orthogonality, we refer the reader to \citep[Section 2]{FV26}, \citep[Section 23.1]{Ism05}, and \citep[Section 1]{VA20}.

Among the most extensively studied families of MOP are those associated with classical weights \citep{ABVA03}: Hermite, Laguerre, and Jacobi. This framework leads, for instance, to the multiple Hermite polynomials \citep{BK05}, which find applications in random matrices \citep{BK04} and non-intersecting Brownian motions \citep{DK07}; multiple Laguerre polynomials \citep{BK05}; and the special case of multiple Jacobi polynomials, which we describe in more detail below.

The Jacobi weight for $\alpha,\beta>-1$ is given by $(1-x)^\alpha(1+x)^\beta$, defined on the interval $[-1,1]$. The main ways of constructing multiple Jacobi polynomials are:
\begin{itemize}
    \item Using Jacobi weights defined on $r$ disjoint intervals, which gives rise to the Jacobi--Angelesco polynomials \citep{DS16}, \citep[Section 23.3.1]{Ism05}.
    
    \item Assigning $r$ distinct values $\alpha_1,\dots,\alpha_r$ to the parameter $\alpha$ and mapping the interval $[-1,1]$ to $[0,1]$. This procedure yields the Jacobi--Piñeiro polynomials \citep{BCVA05,Pin87}. 
\end{itemize}

Throughout this document, we focus on the latter family. As mentioned, Jacobi polynomials are defined as orthogonal polynomials with respect to the weight $(1-x)^\alpha(1+x)^\beta$, $x\in[-1,1]$, but, up to a linear transformation, we might consider the weight $\omega^{(\alpha,\beta)}(x)=x^\alpha(1-x)^\beta$ with $x\in[0,1]$. Then, employing $r$ weight functions of this type but with different values of $\alpha$, we define
\begin{equation}
\label{eq:JP-weights}
    \omega_j(x) = x^{\alpha_j} (1-x)^\beta, \qquad x\in[0,1],\quad j=1,\dots,r,
\end{equation}
where $\alpha_j, \beta > -1$. To guarantee the existence and uniqueness of the corresponding MOP for every multi-index $\vec n\in\NN^r$, it is necessary to impose that $\alpha_i - \alpha_j \notin \ZZ$ whenever $i\neq j$ \citep{ADL23}. In this setting, for a multi-index $\vec n=(n_1,\dots,n_r)$, the type~II Jacobi--Piñeiro polynomial $P^{(\vec \alpha,\beta)}_{\vec n}(x)$ is monic and satisfies the orthogonality conditions \eqref{eq:orthogonality-1-var} \citep[Section 3.7]{VA20}.

These polynomials have been widely studied ever since the emergence of multiple orthogonality. Several representations in terms of hypergeometric functions can be found in \citep{ABVA03, BCVA05}, as well as applications in rational approximation, random walks, number theory, urn models, among others \citep{BDFMAF23, BDFM25, Fis04, GI21, MFOSL22, MFVA16}. For our purposes, the most convenient representation is given by the Rodrigues formula \citep{ABVA03}:
\begin{equation}
    \label{eq:rodrigues-1-var}
    P^{(\vec \alpha, \beta)}_{\vec n}(x) = \left(\dfrac{1}{\omega_r(x)}\dfrac{d^{\,n_r}}{dx^{n_r}}\omega_r(x)\right)\cdots\left(\dfrac{1}{\omega_1(x)}\dfrac{d^{\,n_1}}{dx^{n_1}}\omega_1(x)\right)\phi^{|\vec n|}(x),
\end{equation}
where $\phi(x)=x(x-1)$. Note that \eqref{eq:rodrigues-1-var} can be written as the composition of the $r$ Rodrigues operators
\begin{equation}
    \label{eq:Rodrigues-operator-1-var}
    D_j[f]=\dfrac{1}{\omega_j(x)}\dfrac{d^{n_j}}{dx^{n_j}}(\omega_j(x)f(x)), \qquad j =1,\dots,r,
\end{equation}
applied to a polynomial. In fact, \eqref{eq:Rodrigues-Jacobi}, up to a linear transformation, is a particular case of \eqref{eq:rodrigues-1-var} with $r=1$ and $n_1=n$.

Moreover, Jacobi--Piñeiro MOP also admit a representation in terms of hypergeometric functions ${}_p F_q$ \citep[Theorem 3.2]{BCVA05}:
\begin{equation}
    \label{eq:JP-hypergeometric}
    P_{\vec n}^{\vec\alpha,\beta}(x)=\dfrac{(\vec\alpha+\vec 1)_{\vec n}}{\vec n!}(1-x)^{-\beta}{}_{r+1}F_r\left(\begin{array}{c}
         -\beta-|\vec n|,\vec\alpha+\vec n+\vec 1 \\ \vec\alpha+\vec 1
    \end{array}\Bigg|\, x\right),
\end{equation}
where we use the following notations
\begin{align}
    \label{eq:notations-pochammer}
    \vec 1 &=(1,\dots,1)\in\mathbb N^r,& \vec n!& = n_1!\,\cdots\, n_r!,& (\vec a)_{\vec n}=(a_1)_{n_1}\,\cdots\,(a_r)_{n_r},
\end{align}
for $\vec a=(a_1,\dots,a_r)\in\RR^r$ and $(a)_n=a (a+1)\,\cdots\, (a+n-1)$ is the Pochhammer symbol.

Concerning the possible applications of multiple orthogonal polynomials, and in particular Jacobi--Piñeiro polynomials, it is essential to mention the Hermite--Padé rational approximation problem. This problem extends the classical Padé approximation problem, whose main objective is to construct rational approximations to functions given by Stieltjes transforms 
\begin{equation}
\label{eq:f_j}
f_j(z)=\displaystyle\int_\RR \dfrac{\mathrm d\mu_j(x)}{z-x},    
\end{equation}
with a common denominator \citep[Chapter 4]{NS91}. More precisely, for a multi-index $\vec n =(n_1,\dots,n_r)$, the type~II Hermite--Padé approximation (see \citep[Sections 2.1 and 2.2]{VA06} for type~I Hermite--Padé approximation) consists of finding a polynomial $P_{\vec n}$ of degree $|\vec n|$ and $r$ polynomials $Q_{\vec n,1},\dots,Q_{\vec n,r}$ such that
\begin{equation}
    \label{eq:H-P-univariate}
f_j(z)P_{\vec n}(z) = Q_{\vec n,j}(z)+\mathcal O(z^{-n_j-1}) , \qquad z\to\infty,\quad j=1,\dots,r.
\end{equation}
The solution to this problem is provided by the type~II multiple orthogonal polynomial $P_{\vec n}$ for the system of measures $\mu_1,\dots,\mu_r$, and
\begin{equation*}
Q_{\vec n,j}(z)=\displaystyle\int_\RR \dfrac{P_{\vec n}(z)-P_{\vec n}(x)}{z-x}\,\mathrm d\mu_j(x).
\end{equation*}
In the Jacobi--Piñeiro case, \textit{i.e.}, when $\mathrm d\mu_j(x)=x^{\alpha_j}(1-x)^\beta\,\mathrm dx$ on $[0,1]$, this problem produces rational approximations to certain values of the Riemann zeta function $\zeta(z)$. In particular, this approach is closely related to results establishing the irrationality of $\zeta(3)$ \citep{Ape79}, that infinitely many $\zeta(2k+1)$ are irrational \citep{BR01}, and that at least one among $\zeta(5),\zeta(7),\zeta(9),\zeta(11)$ is irrational \citep{Zud04}.

Having established these preliminaries on MOP and their applications in the univariate setting, we now state the objectives of this paper. The first objective is to address some of the gaps identified in Sections~\ref{sec:intro} and \ref{sec:preliminaries} by introducing a family of bivariate multiple orthogonal polynomials defined through a Rodrigues-type formula, analogous to \eqref{eq:Rodrigues-2-var} and \eqref{eq:rodrigues-1-var}. These polynomials provide a natural bivariate counterpart of the classical Jacobi--Piñeiro family.

Motivated by the connection between multiple orthogonality and Hermite--Padé approximation, our second objective is to formulate and study a Hermite--Padé approximation problem in two variables. More precisely, we construct approximants for two functions, analogous to \eqref{eq:f_j}, associated with two Jacobi measures on the triangle (see \eqref{eq:weights}). These approximants share a common denominator, namely the family of bivariate Jacobi--Piñeiro polynomials introduced in Section~\ref{sec:JP}. In this way, we establish a bivariate analogue of the classical relationship between Jacobi--Piñeiro polynomials and Hermite--Padé approximation.

\section{Jacobi--Piñeiro polynomials on the triangle}\label{sec:JP}

As described in Section \ref{sec:preliminaries}, the type II Jacobi--Piñeiro polynomials $P^{(\vec \alpha,\beta)}_{\vec n}(x)$ are multiple orthogonal with respect to the weight functions $\omega_j(x)$ given in \eqref{eq:JP-weights}, and can be expressed via a Rodrigues formula \eqref{eq:rodrigues-1-var}.

In this section, we consider a bivariate extension of these polynomials on the triangular domain 
\begin{equation}
\tag{\ref{eq:triangle}}
    T = \{(x,y)\in\RR^2: x\geq 0, y\geq 0, x+y \leq 1\},
\end{equation}
often also referred to as the \textit{simplex}, and provide a representation of these polynomials based on a Rodrigues-type formula.

Consider $\alpha_1,\dots,\alpha_r,\beta_1,\dots,\beta_r,\gamma>-1$ with $(\alpha_i,\beta_i)\ne(\alpha_j,\beta_j)$ if $i\ne j$ and define $r\geq 2$ weight functions supported on $T$ and their respective measures: 
\begin{equation}
    \label{eq:weights}
    W_j(x,y) = x^{\alpha_j}y^{\beta_j}(1-x-y)^{\gamma}, \qquad \dr\mu_j(x,y)=W_j(x,y)\,\dr x\,\dr y, \qquad j=1,\dots,r.
\end{equation}
Next, we introduce a family of Rodrigues-type differential operators, extending \eqref{eq:Rodrigues-operator-1-var},  which will be used throughout the remainder of this section.

\begin{definition}
    For every $j\in\{1,\dots,r\}$, consider the weight $W_j(x,y)$ given in \eqref{eq:weights} and two non-negative integers $0\leq k_j\leq n_j$. Given a function $f\in C^\infty(\mathbb R^2)$, we define the operator
    \begin{equation}
    \label{eq:operator-D_j}
        D_{j}[f(x,y)]:=\dfrac{1}{W_j(x,y)}\dfrac{\partial^{n_j}}{\partial x^{n_j-k_j}y^{k_j}}\left(W_j(x,y) x^{n_j-k_j} y^{k_j} f(x,y)\right).
    \end{equation}
\end{definition}

Motivated by the Rodrigues formula for Jacobi--Piñeiro polynomials \eqref{eq:rodrigues-1-var}, we define the following bivariate function:

\begin{definition} Consider $r$ pairs of non-negative integers $0\leq k_j \leq n_j$, $j=1,\dots,r$, and denote $\vec n =(n_1,\dots,n_r)$, $\vec k=(k_1,\dots,k_r)$, $\vec\alpha=(\alpha_1,\dots,\alpha_r)$, $\vec\beta=(\beta_1,\dots,\beta_r)$. Define the function
\begin{equation}
\label{eq:pol-definition-r-measures}
U^{(\vec \alpha,\vec \beta,\gamma)}_{\vec n,\vec k}(x,y) := D_r\circ\cdots\circ D_1[(1-x-y)^{n_1+\cdots +n_r}]
\end{equation}
Equivalently, expanding the composition,
\begin{equation*}
\begin{split}
	U^{(\vec \alpha,\vec \beta,\gamma)}_{\vec n,\vec k}(x,y)  = \dfrac{1}{W_r(x,y)} &\frac{\partial^{n_r}}{\partial x^{n_r-k_r}\partial y^{k_r}}\Big(x^{n_r-k_r+\alpha_r-\alpha_{r-1}}y^{k_r+\beta_r-\beta_{r-1}}  \cdots \times \\ & \times  \frac{\partial^{n_1}}{\partial x^{n_1-k_1}\partial y^{k_1}}(x^{n_1-k_1+\alpha_1}y^{k_1+\beta_1}(1-x-y)^{\gamma+n_1+\cdots+n_r})\Big).\\
\end{split}
\end{equation*}
\end{definition}

Observe that, for a fixed multi-index $\vec n$, the family of polynomials
$U^{(\vec\alpha,\vec\beta,\gamma)}_{\vec n,\vec k}$ is indexed by all admissible
choices of $\vec k$, which determine the distribution of derivatives with respect to $x$ and $y$ in the Rodrigues formula.

Next, we show that this function is a polynomial of degree $|\vec n|=n_1+\cdots+n_r$ in the variables $x$ and $y$, and that it satisfies certain orthogonality conditions with respect to the weights $W_1,\dots, W_r$ on the triangle $T$. First, let us introduce the following result:

\begin{lemma}
    \label{lemma:operator-degree}
    Given $l,m,n\geq0$ with $n\geq l+m$, consider $x^l y^m (1-x-y)^{n-l-m}$, a polynomial of degree $n$. For a fixed $j$, let $D_j$ be the operator introduced in \eqref{eq:operator-D_j}. If $n\ge n_j$, then $D_j[x^l y^m (1-x-y)^{n-l-m}]$ is a linear combination of polynomials of degree at most $n$, therefore a polynomial of degree at most $n$.
\end{lemma}
\begin{proof}
    Splitting the derivatives with respect to $x$ and $y$:
    $$
    \begin{aligned}
        &D_j[x^l y^m(1-x-y)^{n-l-m}]\\ &= x^{-\alpha_j}y^{-\beta_j}(1-x-y)^{-\gamma}\dfrac{\partial^{n_j}}{\partial x^{n_j-k_j}y^{k_j}}\left( x^{\alpha_j+n_j-k_j+l} y^{\beta_j+k_j+m} (1-x-y)^{\gamma+n-l-m}\right) \\ & =
    x^{-\alpha_j}y^{-\beta_j}(1-x-y)^{-\gamma}\dfrac{\partial^{n_j-k_j}}{\partial x^{n_j-k_j}}\left(x^{\alpha_j+n_j-k_j+l}\dfrac{\partial^{k_j}}{\partial y^{k_j}} (y^{\beta_j+k_j+m} (1-x-y)^{\gamma+n-l-m})\right).
    \end{aligned}
    $$
    Focus on the first derivative and apply Leibniz's rule:
    \begin{equation}
        \begin{aligned}
            \dfrac{\partial^{k_j}}{\partial y^{k_j}} (y^{\beta_j+k_j+m} (1-x-y)^{\gamma+n-l-m}) &= \sum_{i=0}^{k_j}\binom{k_j}{i} \dfrac{\partial^{k_j-i}}{\partial y^{k_j-i}}y^{\beta_j+k_j+m} \, \dfrac{\partial^{i}}{\partial y^{i}}(1-x-y)^{\gamma+n-l-m} \\ &=\sum_{i=0}^{k_j}c(i)\, y^{\beta_j+m+i}(1-x-y)^{\gamma+n-l-m-i}
        \end{aligned}
    \end{equation}
    where 
    \begin{equation}
        \label{eq:coef-c-1}
        c(i):=(-1)^i\binom{k_j}{i}(\beta_j+m+i+1)_{k_j-i}(\gamma+n-l-m-i+1)_i, \qquad 0\leq i\leq k_j.
    \end{equation}
    Multiplying by $x^{\alpha_j+n_j-k_j+l}$ and computing the derivative with respect to $x$ we obtain:
    \begin{multline*}
        \dfrac{\partial^{n_j-k_j}}{\partial x^{n_j-k_j}}\left(x^{\alpha_j+n_j-k_j+l}\dfrac{\partial^{k_j}}{\partial y^{k_j}} (y^{\beta_j+k_j+m} (1-x-y)^{\gamma+n-l-m})\right) \\ =\sum_{i=0}^{k_j}c(i)  y^{\beta_j+m+i}\dfrac{\partial^{n_j-k_j}}{\partial x^{n_j-k_j}}x^{\alpha_j+n_j-k_j+l}(1-x-y)^{\gamma+n-l-m-i}.
    \end{multline*}
    Again, by Leibniz's rule,
    \begin{multline*}
        \dfrac{\partial^{n_j-k_j}}{\partial x^{n_j-k_j}}x^{\alpha_j+n_j-k_j+l}(1-x-y)^{\gamma+n-l-m-i}=\sum_{p=0}^{n_j-k_j}(-1)^p\binom{n_j-k_j}{p}(\alpha_j+l+p+1)_{n_j-k_j-p}\times \\ \times (\gamma+n-l-m-i-p+1)_{p}\, x^{\alpha_j+l+p}(1-x-y)^{\gamma+n-l-m-i-p}.
    \end{multline*}
    If we denote 
    \begin{equation*}
        c_j(p,i;l,m):=c(i)(-1)^p\binom{n_j-k_j}{p}(\alpha_j+l+p+1)_{n_j-k_j-p}(\gamma+n-l-m-i-p+1)_{p},
    \end{equation*}
    this expression can be simplified to:
    \begin{equation}
        \label{eq:coef-c-2}
        c_j(p,i;l,m)= (\alpha_j+l+1)_{n_j-k_j} (\beta_j+m+1)_{k_j}(-\gamma-n+l+m)_{p+i} (-1)^{i+p}\frac{(-k_j)_i}{i!(\beta_j+m+1)_i} \frac{(-n_j+k_j)_p}{p!(\alpha_j+l+1)_p}
    \end{equation}
    for $0\leq p\leq n_j-k_j$, $0\leq i \leq k_j$,  and we get
    $$
    D_j[x^l y^m(1-x-y)^{n-l-m}] =  x^{-\alpha_j}y^{-\beta_j}(1-x-y)^{-\gamma} \sum_{p=0}^{n_j-k_j}\sum_{i=0}^{k_j} c_j(p,i;l,m) x^{\alpha_j+l+p}y^{\beta_j+m+i}(1-x-y)^{\gamma+n-l-m-i-p} 
    $$
    and consequently,
    \begin{equation}
    \label{eq:operator-expression}
       D_j[x^l y^m(1-x-y)^{n-l-m}] = \sum_{p=0}^{n_j-k_j}\sum_{i=0}^{k_j} c_j(p,i;l,m) x^{l+p}y^{m+i}(1-x-y)^{n-l-m-i-p},
    \end{equation}
    which is a linear combination of polynomials of degree at most $n$, \textit{i.e.}, a polynomial of degree at most $n$.
\end{proof}

A consequence of the previous Lemma is the following result

\begin{proposition}
\label{propo:degree}
    The function $U_{\vec n,\vec k}^{(\vec\alpha,\vec\beta,\gamma)}(x,y)$ defined in \eqref{eq:pol-definition-r-measures} is a polynomial of degree exactly $|\vec n|=n_1+\dots+n_r$.
\end{proposition}
\begin{proof}
    Since $U_{\vec n,\vec k}^{(\vec\alpha,\vec\beta,\gamma)}(x,y)=D_r\circ\cdots\circ D_1[(1-x-y)^{|\vec n|}]$, applying Lemma~\ref{lemma:operator-degree} $r$ times it is clear that the degree of the polynomial $U_{\vec n,\vec k}^{(\vec\alpha,\vec\beta,\gamma)}(x,y)$ is at most $|\vec n|$. We have to prove that, indeed, the degree is exactly $|\vec n|$.
    Using formula \eqref{eq:operator-expression} $r$ times, we get:
    \begin{equation}\label{eq:Expr_U}
    U_{\vec n,\vec k}^{(\vec\alpha,\vec\beta,\gamma)}(x,y)=\sum_{\vec p=0}^{\vec n-\vec k} \sum_{\vec i=0}^{\vec k} \Big(\prod_{j=1}^r c_j(p_j,i_j;P_j,I_j)\Big)x^{|\vec p|} y^{|\vec i|}(1-x-y)^{|\vec n|-|\vec p|-|\vec i|}
    \end{equation}
    where 
    $$
    \sum_{\vec p=0}^{\vec n-\vec k}= \sum_{p_1=0}^{n_1-k_1}\cdots \sum_{p_r=0}^{n_r-k_r} \,,\qquad 
    \sum_{\vec i=0}^{\vec k}= \sum_{i_1=0}^{k_1}\cdots \sum_{i_r=0}^{k_r}\,, \qquad P_j=\sum_{\ell=1}^{j-1} p_\ell, \qquad  I_j=\sum_{\ell=1}^{j-1} i_\ell\,,
    $$
    with $P_1=0$ and $I_1=0$. If we expand $(1-x-y)^{|\vec n|-|\vec p|-|\vec i|}$ we obtain:
    $$
    (1-x-y)^{|\vec n|-|\vec p|-|\vec i|}=\sum_{\nu=0}^{|\vec n|-|\vec p|-|\vec i|} (-1)^{|\vec n|-|\vec p|-|\vec i|} \frac{(|\vec n|-|\vec p|-|\vec i|)!}{\nu! (|\vec n|-|\vec p|-|\vec i|-\nu)!}x^{\nu} y^{|\vec n|-|\vec p|-|\vec i|-\nu} + \text{lower degree terms}
    $$
    and substituting in \eqref{eq:Expr_U},
    $$
    U_{\vec n,\vec k}^{(\vec\alpha,\vec\beta,\gamma)}(x,y)=\sum_{\vec p=0}^{\vec n-\vec k} \sum_{\vec i=0}^{\vec k} \sum_{\nu=0}^{|\vec n|-|\vec p|-|\vec i|} (-1)^{|\vec n|-|\vec p|-|\vec i|} \frac{(|\vec n|-|\vec p|-|\vec i|)!}{\nu! (|\vec n|-|\vec p|-|\vec i|-\nu)!} \Big(\prod_{j=1}^r c_j(p_j,i_j;P_j,I_j)\Big)x^{|\vec p|+\nu} y^{|\vec n|-|\vec p|-\nu}+ \text{l.d.t}
    $$
    We need to prove that there is at least one coefficient of a monomial of degree $|\vec n|$ that does not vanish. Let us examine the coefficient of $y^{|\vec n|}$. It is obtained when $\nu=0$ and $|\vec p|=0$, that is, $p_1=\dots=p_r=0$:
    $$
    \sum_{\vec i=0}^{\vec k}  (-1)^{|\vec n|-|\vec i|}  \prod_{j=1}^r c_j(0,i_j;0,I_j)
    $$
    Using the expression of $c_j(0,i_j;0,I_j)$ in  \eqref{eq:coef-c-2}, the coefficient simplifies as
    $$
    \sum_{\vec i=0}^{\vec k} (-1)^{|\vec n|} (\gamma+|\vec n|-|\vec i|+1)_{|\vec i|}\prod_{j=1}^r (\alpha_j+1)_{n_j-k_j}(\beta_j+I_j+1)_{k_j} \frac{(-k_j)_{i_j}}{i_j! (\beta_j+I_j+1)_{i_j}}
    $$
    that is,
    $$
    \sum_{\vec i=0}^{\vec k} (-1)^{|\vec n|} (\gamma+|\vec n|-|\vec i|+1)_{|\vec i|}\prod_{j=1}^r (\alpha_j+I_j+1)_{n_j-k_j}(\beta_j+1)_{k_j} \frac{(k_j-i_j+1)_{i_j}}{i_j! (\beta_j+1)_{i_j}}
    $$
    Since all the terms in the previous sum are nonzero and have sign $(-1)^{|\vec n|}$, the coefficient cannot vanish.

    In fact, none of the coefficients of the monomials of degree $|\vec n|$ vanishes. With a similar procedure, we can compute the coefficient of $x^k y^{|\vec n|-k}$ and it is equal to 
    $$
    \begin{aligned}
    \sum_{\vec p=0}^{\vec n-\vec k}\sum_{\vec i=0}^{\vec k} (-1)^{|\vec n|} & \frac{(|\vec n|-|\vec p|-|\vec i|)!}{(k-|\vec p|)!(|\vec n|-|\vec i|-k)!} (\gamma+|\vec n|-|\vec p|-|\vec i|+1)_{|\vec p|+|\vec i|}
    \\
    &\times \prod_{j=1}^r (\alpha_j+P_j+1)_{n_j-k_j}(\beta_j+I_j+1)_{k_j} \frac{(k_j-i_j+1)_{i_j}}{i_j! (\beta_j+I_j+1)_{i_j}}
    \frac{(n_j-k_j-p_j+1)_{p_j}}{p_j! (\alpha_j+P_j+1)_{p_j}}
    \end{aligned}
    $$
    that is, up to $(-1)^{|\vec n|}$, a sum of positive quantities and therefore nonzero. 
\end{proof}

So far, we have introduced the function $U_{\vec n,\vec k}^{(\vec\alpha,\vec\beta,\gamma)}$ as a concatenation of the operators $D_j$, $j=1,\dots, r$, proving it to be a polynomial of degree $|\vec n|$. 

Next, we show the commutativity of the operators $D_j$, $j=1,\dots, r$, which leads to the symmetry of the construction given in \eqref{eq:pol-definition-r-measures}. In other words, the polynomial does not change by changing the order of the composition of the operators.

\begin{proposition}
    \label{propo:D_j-commute}
    Given $i,j\in \{1,\dots,r\}$, the operators $D_i$ and $D_{j}$ commute, \textit{i.e.}:
    \begin{equation}
        \label{eq:symmetry}
        D_i\circ D_j[f(x,y)] = D_j\circ D_i[f(x,y)] 
    \end{equation}
    for every analytic function $f:\RR^2\longrightarrow\RR$.
\end{proposition}
\begin{proof}
    According to \eqref{eq:operator-D_j}, we have that
    \begin{multline*}
         D_i\circ D_j[f(x,y)] = x^{-\alpha_i} y^{-\beta_i}(1-x-y)^{-\gamma} \frac{\partial^{n_i}}{\partial x^{n_i-k_i}\partial y^{k_i}}\Big(x^{n_i-k_i+\alpha_i-\alpha_j}y^{k_i+\beta_i-\beta_j} \times\\ \times \frac{\partial^{n_j}}{\partial x^{n_j-k_j}\partial y^{k_j}}\left(x^{n_j-k_j+\alpha_j}y^{k_j+\beta_j}(1-x-y)^{\gamma}f(x,y)\right)\Big)
    \end{multline*}
    Now, consider the formal Laurent expansion of the function $(1-x-y)^{\gamma}f(x,y)$ as
    \begin{equation}
        (1-x-y)^\gamma f(x,y)=\sum_{\nu,\sigma=-\infty}^{\infty} c_{\nu,\sigma}x^\nu y^\sigma.
    \end{equation}
    In this way, we have
    \begin{multline*}
        D_i\circ D_j[f(x,y)] =(1-x-y)^{-\gamma}\sum_{\nu,\sigma=-\infty}^{\infty} c_{\nu,\sigma} x^{-\alpha_i} y^{-\beta_i}  \frac{\partial^{n_i}}{\partial x^{n_i-k_i}\partial y^{k_i}}\Big(x^{n_i-k_i+\alpha_i-\alpha_j}y^{k_i+\beta_i-\beta_j} \times \\ \times  \frac{\partial^{n_j}}{\partial x^{n_j-k_j}\partial y^{k_j}}\left(x^{n_j-k_j+\alpha_j+\nu}y^{k_j+\beta_j+\sigma}\right)\Big)
    \end{multline*}
    Let us focus on the term
    $$
    x^{-\alpha_i} y^{-\beta_i}  \frac{\partial^{n_i}}{\partial x^{n_i-k_i}\partial y^{k_i}}\Big(x^{n_i-k_i+\alpha_i-\alpha_j}y^{k_i+\beta_i-\beta_j} \frac{\partial^{n_j}}{\partial x^{n_j-k_j}\partial y^{k_j}}\left(x^{n_j-k_j+\alpha_j+\nu}y^{k_j+\beta_j+\sigma}\right)\Big).
    $$
    Computing the derivative inside the parentheses, we get
    $$
    x^{-\alpha_i} y^{-\beta_i} (\alpha_j+\nu+1)_{n_j-k_j}(\beta_j+\sigma+1)_{k_j} \frac{\partial^{n_i}}{\partial x^{n_i-k_i}\partial y^{k_i}}\Big(x^{n_i-k_i+\alpha_i+\nu}y^{k_i+\beta_i+\sigma} \Big),
    $$
    and, computing the remaining derivative
    $$
    (\alpha_j+\nu+1)_{n_j-k_j}(\beta_j+\sigma+1)_{k_j}(\alpha_i+\nu+1)_{n_i-k_i}(\beta_i+\sigma+1)_{k_i}\, x^\nu y^\sigma.
    $$
    Therefore, we finally get
    \begin{multline}
    \label{eq:kampe}
         D_i\circ D_j[f(x,y)] = (1-x-y)^{-\gamma}\times  \\ \times \sum_{\nu,\sigma=-\infty}^{\infty} c_{\nu,\sigma} (\alpha_j+\nu+1)_{n_j-k_j}(\beta_j+\sigma+1)_{k_j}(\alpha_i+\nu+1)_{n_i-k_i}(\beta_i+\sigma+1)_{k_i}\, x^\nu y^\sigma .
    \end{multline}
    Since the final expression is symmetric under the interchange of $i$ and $j$, equation \eqref{eq:symmetry} follows immediately.

\end{proof}

As a consequence, the bivariate Jacobi--Piñeiro polynomials \eqref{eq:pol-definition-r-measures} satisfy the following property:

\begin{corollary}
    \label{cor:symmetry}
    Consider $r$ pairs of non-negative integers $0\leq k_j \leq n_j$, $j=1,\dots,r$ and denote $\vec n =(n_1,\dots,n_r)$, $\vec k=(k_1,\dots,k_r)$. Let $\pi:\{1,\dots,r\}\longrightarrow\{1,\dots,r\}$ be a permutation of $\{1,\dots,r\}$. Then, the following symmetry condition holds
    $$U_{\vec n,\vec k}^{(\vec\alpha,\vec\beta,\gamma)}(x,y)=D_r\circ\cdots\circ D_1[(1-x-y)^{|\vec n|}]=D_{\pi(1)}\circ\cdots\circ D_{\pi(r)}[(1-x-y)^{|\vec n|}].$$
\end{corollary}
    
This means that the polynomial $U_{\vec n,\vec k}^{(\vec\alpha,\vec\beta,\gamma)}$ is related to the weights $W_1,\dots,W_r$ themselves and not their order. 

Regarding the proof of Proposition~\ref{propo:D_j-commute} and taking into account the formal series expansion

\begin{equation}
    \label{eq:expansion}
    \begin{aligned}
        (1-x-y)^{n+\gamma}&= \sum_{\nu=0}^{\infty}\sum_{\sigma=0}^\infty (-1)^{\nu+\sigma}\dfrac{\Gamma(n+\gamma+1)}{\Gamma(n+\gamma-\nu-\sigma+1)} \dfrac{x^\nu y^\sigma}{\nu!\,\sigma!}\\
        &=\sum_{\nu=0}^{\infty}\sum_{\sigma=0}^\infty (-n-\gamma)_{\nu+\sigma}\,\dfrac{x^\nu y^\sigma}{\nu!\,\sigma!} ,
    \end{aligned}
\end{equation}
which is absolutely convergent whenever $(x,y)\in T$ (see \citep[Chapter 16]{OLBC10}),
an alternative expression for these polynomials can be obtained,
similarly to the case of the type II Jacobi--Pi\~neiro polynomials, which
admit a hypergeometric representation after factoring out an appropriate term, see \eqref{eq:JP-hypergeometric} or  \citep[Theorem 3.5]{BCVA05}.

We introduce the Kampé de Fériet hypergeometric function (see, e.g., \citep{AK26, Lau1893}).
\begin{multline}
	\label{MKF}
	\KF{p:n_1;\cdots;n_r}{q:m_1;\cdots;m_r}{(a_1,\ldots,a_p):(b^1_1,\ldots,b^1_{n_1});\cdots;(b^r_1,\ldots,b^r_{n_r})}{(\alpha_1,\ldots,\alpha_q):(\beta^1_1,\ldots,\beta^1_{m_1});\cdots;(\beta^r_1,\ldots,\beta^r_{m_r})}{x_1,\ldots,x_r}\\
	\coloneq\sum_{l_1=0}^{\infty}\cdots\sum_{l_r=0}^{\infty}\dfrac{(a_1)_{l_1+\cdots+l_r}\cdots(a_p)_{l_1+\cdots+l_r}}{(\alpha_1)_{l_1+\cdots+l_r}\cdots(\alpha_q)_{l_1+\cdots+l_r}}\dfrac{(b^1_1)_{l_1}\cdots(b^1_{n_1})_{l_1}}{(\beta^1_1)_{l_1}\cdots(\beta^1_{m_1})_{l_1}}\cdots\dfrac{(b^r_1)_{l_r}\cdots(b^r_{n_r})_{l_r}}{(\beta^r_1)_{l_r}\cdots(\beta^r_{m_r})_{l_r}}\dfrac{x_1^{l_1}}{l_1!}\cdots\dfrac{x_r^{l_r}}{l_r!},
\end{multline}
which is an extension of the well-known hypergeometric series ${}_p F_q$. In the next result, we employ Proposition~\ref{propo:D_j-commute} together with the expansion \eqref{eq:expansion} to obtain a hypergeometric expression of $U_{\vec n,\vec k}^{(\vec\alpha,\vec\beta,\gamma)}$.

\begin{proposition}\label{prop:hypergeometric-expression} Consider $r$ pairs of non-negative integers $0\leq k_j \leq n_j$, $j=1,\dots,r$, and denote $\vec n =(n_1,\dots,n_r)$, $\vec k=(k_1,\dots,k_r)$, $\vec\alpha=(\alpha_1,\dots,\alpha_r)$, $\vec\beta=(\beta_1,\dots,\beta_r)$, $\vec 1 =(1,\dots,1)\in\mathbb R^r$. Then, the polynomial $U_{\vec n,\vec k}^{(\vec\alpha,\vec\beta,\gamma)}$ introduced in \eqref{eq:pol-definition-r-measures} admits the following representation
\begin{multline}
    U_{\vec n,\vec k}^{(\vec\alpha,\vec\beta,\gamma)}(x,y)=(1-x-y)^{-\gamma}\Big(\prod_{j=1}^r (\alpha_j+1)_{n_j-k_j}(\beta_j+1)_{k_j}\Big)\times\\\times\KF{1:r;r}{0:r;r}{-|\vec n|-\gamma:(\alpha_1+n_1-k_1+1,\dots,\alpha_r+n_r-k_r+1);(\beta_1+k_1+1,\dots,\beta_r+k_r+1)}{-:(\alpha_1+1,\dots,\alpha_r+1);(\beta_1+1,\dots,\beta_r+1)}{x, y},
\end{multline}
or using a more compact notation \eqref{eq:notations-pochammer}
\begin{equation}\label{eq:hypergeometric-expansion-r-measures}
        U_{\vec n,\vec k}^{(\vec\alpha,\vec\beta,\gamma)}(x,y)=(1-x-y)^{-\gamma}(\vec \alpha+\vec 1)_{\vec n-\vec k}(\vec \beta+\vec 1)_{\vec k}\, \KF{1:r;r}{0:r;r}{-|\vec n|-\gamma:\vec\alpha+\vec n-\vec k+\vec 1;\vec \beta+\vec k+\vec 1}{-:\vec\alpha+\vec 1;\vec\beta+\vec 1}{x, y}.
\end{equation}
\end{proposition}
\begin{proof}

Regarding Proposition~\ref{propo:D_j-commute}, without loss of generality, we will start computing $D_1[(1-x-y)^{|\vec n|}]$. By \eqref{eq:expansion} and following the same steps as in the proof of Proposition~\ref{propo:D_j-commute}, we get that
\begin{equation}
    D_1[(1-x-y)^{|\vec n|}]=(1-x-y)^{-\gamma}\sum_{\nu=0}^\infty\sum_{\sigma=0}^{\infty}(-|\vec n|-\gamma)_{\nu+\sigma}(\alpha_1+\nu+1)_{n_1-k_1}(\beta_1+\sigma+1)_{k_1}\dfrac{x^\nu y^\sigma}{\nu!\sigma!}.
\end{equation}
Applying $D_2$ to the latter expression yields \eqref{eq:kampe}:
\begin{multline*}
D_2\circ D_1[(1-x-y)^{|\vec n|}] = (1-x-y)^{-\gamma}\times  \\ \times \sum_{\nu=0}^\infty\sum_{\sigma=0}^{\infty}(-|\vec n|-\gamma)_{\nu+\sigma}(\alpha_1+\nu+1)_{n_1-k_1}(\beta_1+\sigma+1)_{k_1}(\alpha_2+\nu+1)_{n_2-k_2}(\beta_2+\sigma+1)_{k_2}\,\dfrac{x^\nu y^\sigma}{\nu!\sigma!}.
\end{multline*}
Repeating this process $r$ times, we finally obtain
\begin{multline}\label{eq:series-expansion}
D_r\circ\cdots\circ  D_1[(1-x-y)^{|\vec n|}]=(1-x-y)^{-\gamma}\times  \\ \times \sum_{\nu=0}^\infty\sum_{\sigma=0}^{\infty}(-|\vec n|-\gamma)_{\nu+\sigma} \Big(\prod_{j=1}^r (\alpha_j+\nu+1)_{n_j-k_j}(\beta_j+\sigma+1)_{k_j}\Big)\dfrac{x^\nu y^\sigma}{\nu!\sigma!}.
\end{multline}

Next, we use the identities
\begin{align*}
(\alpha_j+\nu+1)_{n_j-k_j}
=
\frac{(\alpha_j+1)_{n_j-k_j}
(\alpha_j+n_j-k_j+1)_\nu}
{(\alpha_j+1)_\nu},
\qquad j=1,\dots,r,
\end{align*}
and
\begin{align*}
(\beta_j+\sigma+1)_{k_j}
=
\frac{(\beta_j+1)_{k_j}
(\beta_j+k_j+1)_\sigma}
{(\beta_j+1)_\sigma},
\qquad j=1,\dots, r.
\end{align*}
Substituting these expressions into the previous series yields the desired result.
    
\end{proof}

Once we know that this construction is a polynomial of degree $|\vec n|$, analogous to \eqref{eq:rodrigues-1-var} but in the bivariate setting, satisfies a desirable symmetry property and admits an hypergeometric representation, it is time to prove its multiple orthogonality conditions.

\begin{proposition}\label{propo:orthogonality}
    Consider $r$ pairs of non-negative integers $0\leq k_j\leq n_j$, $j=1,\dots,r$. Then the polynomial $U^{(\vec \alpha,\vec \beta,\gamma)}_{\vec n,\vec k}(x,y)$ introduced in \eqref{eq:pol-definition-r-measures} satisfies the following orthogonality conditions:
    \begin{equation}
        \label{eq:orthogonality-r-measures}
        \int_T U^{(\vec \alpha,\vec \beta,\gamma)}_{\vec n,\vec k}(x,y) \,x^{n-k}y^k \,W_j(x,y)\,\dr x\dr y=0 \qquad\text{ if } \, 0\leq n-k<n_j-k_j \, \text{ or } \, 0\leq k < k_j,
    \end{equation}
    for $j=1,\dots,r$.
\end{proposition}
\begin{proof}
    Consider $j\in\{1,\dots,r\}$, using the commutation of the operators $D_1,\dots,D_r$ (Corollary~\ref{cor:symmetry}), we can express the polynomial $U^{(\vec \alpha,\vec \beta,\gamma)}_{\vec n,\vec k}(x,y)$ as
    \begin{equation*}
        \begin{aligned}
            U^{(\vec \alpha,\vec \beta,\gamma)}_{\vec n,\vec k}(x,y) &= (D_r\circ\cdots\circ D_j\circ \cdots \circ D_1)[(1-x-y)^{n_1+\cdots+n_r}]\\ &=(D_r\circ\cdots\circ D_{j-1}\circ D_{j+1}\circ \cdots \circ D_1\circ D_j)[(1-x-y)^{n_1+\cdots+n_r}]
        \end{aligned}
    \end{equation*}     
    In this way, the integral we have to compute is
    \begin{equation*}
        \begin{aligned}
    & \int_T U^{(\vec \alpha,\vec \beta,\gamma)}_{\vec n,\vec k}(x,y) x^{n-k}y^k W_j(x,y)\dr x\dr y \\ &=\int_T (D_r\circ\cdots\circ D_{j-1}\circ D_{j+1}\circ \cdots \circ D_1\circ D_j)[(1-x-y)^{n_1+\cdots+n_r}]x^{n-k}y^k W_j(x,y)\dr x\dr y 
    \\ &=
    \int_T \dfrac{1}{W_j(x,y)}\frac{\partial^{n_j}}{\partial x^{n_j-k_j}\partial y^{k_j}}\Big(x^{n_j-k_j+\alpha_j-\alpha_{1}}y^{k_{j}+\beta_j-\beta_{1}}  \cdots \times \\ &\times \frac{\partial^{n_r}}{\partial x^{n_r-k_r}\partial y^{k_r}}(x^{n_r-k_r+\alpha_r}y^{k_r+\beta_r}(1-x-y)^{\gamma+n_1+\cdots+n_r})\Big) x^{n-k}y^k W_j(x,y)\dr x \dr y.
    \end{aligned}
    \end{equation*}
    By applying integration by parts $n_j$ times and using Lemma~\ref{lemma:operator-degree} to note that the terms evaluated on the boundary of $T$ vanish, we obtain
    \begin{multline*}
        (-1)^{n_j}\int_T \Big(x^{n_j-k_j+\alpha_j-\alpha_{1}}y^{k_{j}+\beta_j-\beta_{1}} \cdots \frac{\partial^{n_r}}{\partial x^{n_r-k_r}\partial y^{k_r}}(x^{n_r-k_r+\alpha_r}y^{k_r+\beta_r}(1-x-y)^{\gamma+n_1+\cdots+n_r})\Big) \times \\ \times \frac{\partial^{n_j}}{\partial x^{n_j-k_j}\partial y^{k_j}}(x^{n-k}y^k)\dr x\dr y.
    \end{multline*}
    Finally, observe that $\frac{\partial^{n_j}}{\partial x^{n_j-k_j}\partial y^{k_j}}(x^{n-k}y^k)$ vanishes whenever $n-k<n_j-k_j$ or $k<k_j$, and this argument is valid for $j=1,\dots,r$.
\end{proof}

Once this construction has been introduced and proven to be a multiple orthogonal polynomial in the simplex, we consider the particular case concerning just $2$ measures, allowing us to simplify some notations in order to get an explicit expression of the polynomials. 

\subsection{The two-measure case}
\label{sec:2-measures}

In this section,  we study the case of two measures ($r=2$), so we denote $\vec n = (n_1,n_2)$, $\vec k=(k_1,k_2)$, $\vec\alpha=(\alpha_1,\alpha_2)$ and $\vec\beta=(\beta_1,\beta_2)$. Then, definition \eqref{eq:pol-definition-r-measures} becomes
\begin{equation}
\label{eq:pol-definition}
    \begin{split}
    U_{\vec n, \vec k}^{(\vec \alpha, \vec \beta, \gamma)}(x,y) :=& D_2\circ D_1[(1-x-y)^{n_1+n_2}]  \\ =& 
        \frac{1}{W_2(x,y)} \frac{\partial^{n_2}}{\partial x^{n_2-k_2}\partial y^{k_2}}\Big(x^{n_2-k_2+\alpha_2-\alpha_1}y^{k_2+\beta_2-\beta_1} \times\\ \times  & \frac{\partial^{n_1}}{\partial x^{n_1-k_1}\partial y^{k_1}}\left(x^{n_1-k_1+\alpha_1}y^{k_1+\beta_1}(1-x-y)^{n_1+n_2+\gamma}\right)\Big). 
    \end{split}
\end{equation}
By Corollary~\ref{cor:symmetry}, the order of the operators $D_1$ and $D_2$ is irrelevant. Besides, recall that the degree of this polynomial is $n_1+n_2$ according to Proposition~\ref{propo:degree}.

In the following, we derive an explicit expression of the polynomial by evaluating the composition of the operators.

\begin{proposition}
    Consider two pairs of non-negative integers $0\leq k_1 \leq n_1$ and $0\leq k_2 \leq n_2$. Then, the polynomial $U_{\vec n, \vec k}^{(\vec \alpha, \vec \beta, \gamma)}(x,y)$ defined in \eqref{eq:pol-definition} admits the representation
    \begin{equation}
        \label{eq:explicit-formula}
        U_{\vec n, \vec k}^{(\vec \alpha, \vec \beta, \gamma)}(x,y) = \sum_{l=0}^{n_1+n_2-k_1-k_2}\ \sum_{m=0}^{k_1+k_2} c(l,m) x^l y^m (1-x-y)^{n_1+n_2-l-m}.
    \end{equation}
        where the coefficients $c(l,m)$ are given by:
    $$
    \begin{aligned}     
        \label{eq:coef-c}
        c(l,m)=&(\alpha_1+1)_{n_1-k_1} (\beta_1+1)_{k_1}
        (\alpha_2+1)_{n_2-k_2} (\beta_2+1)_{k_2}  
        \\
         &\times \frac{(-1)^{l+m} (-\gamma-n_1-n_2)_{l+m} (k_2-n_2)_l (-k_2)_m}{l! m! (\alpha_2 +1)_l (\beta_2+1)_m} 
         \\ 
         & \times \sum_{p=0}^l \frac{(-l)_p (k_1-n_1)_p (\alpha_2+1+n_2-k_2)_p}{p! (n_2-k_2-l+1)_p (\alpha_1+1)_p} 
         \sum_{i=0}^m \frac{(-m)_i (-k_1)_i (\beta_2+1+k_2)_i}{i! (k_2-m+1)_i (\beta_1+1)_i}.
    \end{aligned}
    $$
\end{proposition}
\begin{proof}
Use \eqref{eq:Expr_U} and \eqref{eq:coef-c-2} to calculate
\begin{align*}
U_{\vec n, \vec k}^{(\vec \alpha, \vec \beta, \gamma)}(x,y) = \sum_{p_1=0}^{n_1-k_1} \sum_{p_2=0}^{n_2-k_2} \sum_{i_1=0}^{k_1} \sum_{i_2=0}^{n_2} c_2(p_2,i_2,p_1,i_1)c_1(p_1,i_1,0,0) x^{p_1+p_2} y^{i_1+i_2} (1-x-y)^{n_1+n_2-p_1-p_2-i_1-i_2}
\end{align*}
where, after some reductions, we get 
\begin{align*}
c_2(p_2,i_2,p_1,i_1)&c_1(p_1,i_1,0,0) = (\alpha_1+1)_{n_1-k_1} (\beta_1+1)_{k_1} (\alpha_2+1)_{n_2-k_2} (\beta_2+1)_{k_2} (-1)^{p_1+i_1+p_2+i_2} 
\\
 & \times (-\gamma-n_1-n_2)_{p_1+i_1+p_2+i_2}\frac{(k_1-n_1)_{p_1} (-k_1)_{i_1}  }{p_1! i_1!(\alpha_2+1)_{p_1+p_2}(\beta_2+1)_{i_1+i_2}(\alpha_1+1)_{p_1} (\beta_1+1)_{i_1}}  
 \\
 &\times \frac{(k_2-n_2)_{p_2} (-k_2)_{i_2} (\alpha_2+1+n_2-k_2)_{p_1} (\beta_2+1+k_2)_{i_1}}{p_2! i_2!}
\end{align*}
Finally, for $0\leq l \leq n_1+n_2-k_1-k_2$ and $0 \leq m \leq k_1+k_2$, the coefficient $c(l,m)$ in \eqref{eq:coef-c}, the one multiplying the term $x^{l} y^{m}(1-x-y)^{n_1+n_2-l-m}$ is obtained by summing all the terms where $p_1+p_2 = l$ and $i_1+i_2 = m$, and taking into account that \( (k_1-n_1)_{p_1} (-k_1)_{i_1}= 0 \), for \( p_1 >n_1-k_1\) or \(i_1 >k_1\), and \((k_2-n_2)_{p_2} (-k_2)_{i_2}= 0\), for \( p_2 > n_2-k_2\) or \( i_2 > k_2 \).
Therefore,
\begin{multline}
    c(l,m) = (\alpha_1+1)_{n_1-k_1}(\beta_1+1)_{k_1}(\alpha_2+1)_{n_2-k_2}(\beta_2+1)_{k_2} 
 \frac{(-1)^{l+m} (-\gamma-n_1-n_2)_{l+m}}{(\alpha_2+1)_l (\beta_2+1)_m} 
 \\ \times 
\sum_{p=0}^l \frac{(k_1-n_1)_p (k_2-n_2)_{l-p}(\alpha_2+1+n_2-k_2)_p   }{p!(l-p)!(\alpha_1+1)_p}
\sum_{i=0}^m \frac{(-k_1)_i (-k_2)_{m-i}(\beta_2+1+k_2)_i   }{i!(m-i)!(\beta_1+1)_i}
\end{multline}
Now, using
\begin{align*}
    &(k_2-n_2)_{l-p} = \frac{(k_2-n_2)_l (-1)^p}{(n_2-k_2-l+1)_p} \\
    &(-k_2)_{m-i} = \frac{(-k_2)_m (-1)^i}{(k_2-m+1)_i}\\
    &(l-p)! = \frac{(-1)^p l!}{(-l)_p}
    \end{align*}
we obtain \eqref{eq:coef-c}, the stated result.
\end{proof}

In the two-measure case, the representation given in Proposition~\ref{prop:hypergeometric-expression} takes the form:

\begin{multline*}
        U_{\vec n, \vec k}^{(\vec \alpha, \vec \beta, \gamma)}(x,y) =
        (1-x-y)^{-\gamma}(\alpha_1+1)_{n_1-k_1} (\alpha_2+1)_{n_2-k_2}(\beta_1+1)_{k_1}
        (\beta_2 +1)_{k_2}\times \\
        \times
        \sum_{\nu=0}^{\infty}  \sum_{\sigma=0}^{\infty} (-\gamma-n_1-n_2)_{\nu+\sigma}\frac{(\alpha_1+n_1-k_1+1)_\nu (\alpha_2+n_2-k_2+1)_\nu} {(\alpha_1+1)_\nu (\alpha_2+1)_\nu}\times \\
        \times \frac{(\beta_1+k_1+1)_\sigma (\beta_2+k_2+1)_\sigma} {(\beta_1+1)_\sigma (\beta_2+1)_\sigma} \frac{x^\nu y^\sigma}{\nu! \sigma!},
\end{multline*}
      for $(x,y)\in T$ and therefore
 \begin{multline}
        \label{eq:formula_hypergeometrica}
        U_{\vec n, \vec k}^{(\vec \alpha, \vec \beta, \gamma)}(x,y) =
        (1-x-y)^{-\gamma}(\alpha_1+1)_{n_1-k_1} (\alpha_2+1)_{n_2-k_2}(\beta_1+1)_{k_1}
        (\beta_2 +1)_{k_2} \\
        \times
        \KF{1:2;2}{0:2;2}{-\gamma-n_1-n_2:(\alpha_1+n_1-k_1+1,\alpha_2+n_2-k_2+1);(\beta_1+k_1+1,\beta_2+k_2+1)}{-:(\alpha_1+1,\alpha_2+1);(\beta_1+1,\beta_2+1)}{x,y}
        .
\end{multline}
These expressions are particular cases of \eqref{eq:series-expansion} and \eqref{eq:hypergeometric-expansion-r-measures}.

Although the coefficients admit a complicated closed form, the expression \eqref{eq:explicit-formula} allows for a straightforward symbolic implementation. This polynomial will play a fundamental role in Section~\ref{sec:HP}, where it arises as the common denominator of the corresponding bivariate Hermite--Padé-type approximants.

To conclude this section, we present the extension to the multivariate setting, that is, to an arbitrary number of variables $d\geq 2$.

\subsection{The multivariate case}\label{sec:more-than-2-vars}

The bivariate construction described above naturally extends to a higher-dimensional simplex. To simplify the notation, we present the case of two measures ($r=2$), although the general case can be obtained in a completely analogous way.

Let us consider the $d$-dimensional \textit{simplex}
\begin{equation}
\label{eq:simplex-multivariate}
    T^d = \{(x_1,\dots,x_d)\in\RR^d: x_1\geq 0,\dots,x_d\geq 0, x_1+\cdots + x_d\leq 1\}.
\end{equation}
and the multivariate weights
\begin{align*}
     W_1(x_1,\dots,x_d)&=x_1^{\alpha_1}\,\cdots\, x_d^{\alpha_d}(1-|\mathbf x|)^\gamma, & W_2(x_1,\dots,x_d)&=x_1^{\beta_1}\,\cdots\, x_d^{\beta_d}(1-|\mathbf x|)^\gamma,
\end{align*}
where $\vec \alpha = (\alpha_1,\dots,\alpha_d)\neq\vec \beta =(\beta_1,\dots,\beta_d)$, $\alpha_1,\dots,\alpha_d,\beta_1,\dots,\beta_d,\gamma>-1$, $\mathbf x =(x_1,\dots,x_d)\in \mathbb R^d$ and $|\mathbf x|=x_1+\cdots+x_d$ is the $\ell^1$ norm of $\mathbf x$. Now, consider two multi-indices $\mathbf n =(n_1,\dots,n_d)$, $\mathbf m =(m_1,\dots,m_d)\in\NN^d_0$, and define the following operators, extending the ones in \eqref{eq:operator-D_j}:
\begin{equation}
    \begin{array}{c}
     D_1[f(\mathbf x)]:= \dfrac{1}{W_1(\mathbf x)}\dfrac{\partial^{|\mathbf n|}}{\partial x_1^{n_1}\,\cdots \,\partial x_d^{n_d}}\left(x_1^{n_1+\alpha_1}\,\cdots\, x_d^{n_d+\alpha_d}(1-|\mathbf x|)^{\gamma}f(\mathbf x)\right)\\
     D_2[f(\mathbf x)]:= \dfrac{1}{W_2(\mathbf x)}\dfrac{\partial^{|\mathbf m|}}{\partial x_1^{m_1}\,\cdots \,\partial x_d^{m_d}}\left(x_1^{m_1+\beta_1}\,\cdots\, x_d^{m_d+\beta_d}(1-|\mathbf x|)^{\gamma}f(\mathbf x)\right).
    \end{array}
\end{equation}

\begin{proposition}
The function 
\begin{equation*}
    U^{(\vec \alpha,\vec \beta,\gamma)}_{\mathbf n,\mathbf m}(\mathbf x):= D_2\circ D_1[(1-|\mathbf x|)^{|\mathbf n|+|\mathbf m|}]
\end{equation*}
is a polynomial of degree $|\mathbf n|+|\mathbf m|$ and satisfies the following orthogonality conditions
\begin{equation}
\begin{array}{c}
     \displaystyle\int_{T^d} U^{(\vec \alpha,\vec \beta,\gamma)}_{\mathbf n,\mathbf m}(\mathbf x) x_1^{l_1}\,\cdots\, x_d^{l_d}\, W_1(\mathbf x)\, \dr \mathbf x = 0,\qquad \text{ whenever } l_i<n_i \text{ for at least one index } i\in\{1,\dots,d\},\\[7pt]
     \displaystyle\int_{T^d} U^{(\vec \alpha,\vec \beta,\gamma)}_{\mathbf n,\mathbf m}(\mathbf x) x_1^{l_1}\,\cdots\, x_d^{l_d} \,W_2(\mathbf x)\, \dr \mathbf x= 0,\qquad \text{ whenever } l_i<n_i \text{ for at least one index } i\in\{1,\dots,d\}.
\end{array}
\end{equation}
    
\end{proposition}

The proofs follow the same arguments as those used in Propositions~\ref{propo:degree} and \ref{propo:orthogonality}, together with the corresponding multivariate version of Lemma~\ref{lemma:operator-degree}, and are omitted for brevity.

In the multivariate setting, a hypergeometric representation also holds, employing the fact that the Kampé de Fériet function \eqref{MKF} is defined for an arbitrary number of variables. In this way, we introduce the formal series expansion
\begin{equation*}
\begin{aligned}
        (1-|\mathbf x|)^{n+\gamma}&=\sum_{\nu_1=0}^\infty\cdots\sum_{\nu_d=0}^{\infty}(-n-\gamma)_{\nu_1+\cdots+\nu_d}\dfrac{x_1^{\nu_1}\cdots x_d^{\nu_d}}{\nu_1 !\cdots \nu_d!}, \\
        &= \sum_{\vec\nu\in\mathbb N_0^d}(-n-\gamma)_{|\vec \nu|} \dfrac{\mathbf x^{\vec \nu}}{\vec \nu!},\qquad \mathbf x \in T^d,
\end{aligned}
\end{equation*}
where we employed the notation $\vec\nu=(\nu_1,\dots,\nu_d)\in\NN^d_0$, $\mathbf x^{\vec \nu}=x_1^{\nu_1}\cdots x_d^{\nu_d}$ and the vector extensions of factorials and Pochammer symbols \eqref{eq:notations-pochammer}. Applying the operators $D_1$ and $D_2$ to the polynomial $(1-|\mathbf x|)^{|\vec n|+|\vec m|}$ and using the above expansion, we obtain
\begin{equation*}
    \begin{aligned}
        &U^{(\vec \alpha,\vec \beta,\gamma)}_{\mathbf n,\mathbf m}(\mathbf x)= D_2\circ D_1[(1-|\mathbf x|)^{|\mathbf n|+|\mathbf m|}]\\
        &= (1-|\mathbf x|)^{-\gamma} (\vec\alpha+\vec 1)_{\mathbf n}(\vec\beta+\vec 1)_{\mathbf m}\sum_{\vec\nu\in\mathbb N^d_0}(-|\mathbf n|-|\mathbf m|-\gamma)_{|\vec \nu|}\dfrac{(\vec\alpha+\mathbf n +\vec 1)_{\vec\nu}}{(\vec\alpha+\vec 1)_{\vec \nu}}\dfrac{(\vec\beta+\mathbf m +\vec 1)_{\vec\nu}}{(\vec\beta+\vec 1)_{\vec \nu}}\dfrac{\mathbf x^{\vec\nu}}{\vec\nu!}\\
        &=(1-|\mathbf x|)^{-\gamma} (\vec\alpha+\vec 1)_{\mathbf n}(\vec\beta+\vec 1)_{\mathbf m} \times \\ &\qquad \times 
        \KF{1:(2,\overset{(d)}{\dots},2)}{0:(2,\overset{(d)},2)}{-\gamma-|\mathbf n|-|\mathbf m|:(\alpha_1+n_1+1,\beta_1+m_1+1);\cdots;(\alpha_d+n_d+1,\beta_d+m_d+1)}{-:(\alpha_1+1,\beta_1+1);\cdots;(\alpha_d+1,\beta_d+1)}{x_1,\dots,x_d}.
    \end{aligned}
\end{equation*}

This shows that the construction is not intrinsically restricted to the bivariate setting and naturally extends to a higher-dimensional simplex, even in the presence of multiple orthogonality.

\section{Hermite--Pad\'e type approximation}\label{sec:HP}

In this section, we show that the bivariate Jacobi--Pi\~neiro polynomials \eqref{eq:pol-definition-r-measures} give rise to analogues of Hermite--Pad\'e approximants for bivariate functions defined as Stieltjes transforms \citep{AFPP08}. This framework builds upon the multivariate Pad\'e approximation introduced in \citep{Sor02}. Throughout this section, we restrict our attention to the case $r=2$, noting that the derivation extends analogously to the general setting. Let us define the following functions:
\begin{equation}
\label{eq:Ej-integral}
    E_j(z,w):= \int_T \frac{\mathrm d\mu_j(x,y)}{(z - x)(w - y)}, \quad j=1,2,
\end{equation}
where $\mathrm d \mu_j(x,y)=W_j(x,y)\,\mathrm{d} x\,\mathrm{d} y$ is defined in \eqref{eq:weights}. These functions serve as the target function for approximation, playing a role analogous to that of $f_j(z)$ \eqref{eq:f_j} in the univariate setting. This function can be expanded as a double Laurent series at infinity (for sufficiently large $|z|$ and $|w|$ ):
\begin{equation}
    \label{eq:Ej-series}
    E_j(z,w)= \sum_{l=0}^\infty \sum_{m=0}^\infty \frac{c_{l,m}^{(j)}}{z^{l+1} w^{m+1}}, \quad\text{ where } \, c_{l,m}^{(j)} := \int_T x^l y^m W_j(x,y) \mathrm{d} x \mathrm{d} y, \quad j=1,2.
\end{equation}
The expansion follows by applying the geometric series identity 
$$\dfrac{1}{a-b}=\sum_{k=0}^\infty\dfrac{b^k}{a^{k+1}}, \qquad |b|<|a|,
$$
to the factors $(z-x)^{-1}$ and $(w-y)^{-1}$.

Furthermore, the coefficients $c_{l,m}^{(j)}$ can be explicitly evaluated as:
\begin{equation*}
    c_{l,m}^{(j)}=\frac{\Gamma (\gamma +1) \Gamma (\alpha_j+l+1) \Gamma (\beta_j+m+1)}{\Gamma (\alpha_j+\beta_j+\gamma+l+m +3)}, \quad j=1,2.
\end{equation*}

Given $0\leq k_1\leq n_1$ and $0\leq k_2\leq n_2$, let us denote $\vec n=(n_1,n_2)$ and $\vec k = (k_1,k_2)$. Our objective is to formulate and solve a Hermite--Pad\'e-type approximation problem. Specifically, we aim to approximate the functions $E_1(z,w)$ and $E_2(z,w)$ via the constructions 
\begin{equation}
    \label{eq:R-approx-def}
    R_{\vec n,\vec k}^{(j)}(z,w) := \dfrac{\Phi_{\vec n,\vec k}^{(j)}(z,w)}{A_{\vec n,\vec k}(z,w)}
\end{equation}
for $j=1,2$, which share a common denominator. Consequently, the denominator of the approximants must be independent of $j$. 

More precisely, we seek a polynomial $A_{\vec n,\vec k}$ of total degree at most $n_1+n_2$ and $\Phi_{\vec n,\vec k}^{(j)}$ such that 
\begin{equation}
\label{eq:approximation}
    A_{\vec n,\vec k}(z,w)E_j(z,w)=\Phi_{\vec n,\vec k}^{(j)}(z,w) + \mathcal{O}(z^{-(n_j-k_j)-1} w^{-k_j-1}),\qquad \text{ as } z,w\longrightarrow \infty,\qquad j =1,2.
\end{equation}
The exponents are chosen so that the approximation order reflects the orthogonality conditions satisfied by the polynomial
$U^{(\vec\alpha,\vec\beta,\gamma)}_{\vec n,\vec k}$.

To achieve this, consider the integral representation \eqref{eq:Ej-integral}. By introducing the term $A_{\vec n,\vec k}(z,w)$ in both the numerator and denominator, and subsequently adding and subtracting $A_{\vec n,\vec k}(x,y)$ within the integrand, we obtain:
$$
\begin{aligned}
    E_j(z,w)&=\frac{1}{A_{\vec n,\vec k}(z,w)} \int_{T} \frac{A_{\vec n,\vec k}(z,w) \mathrm{d} \mu_j (x,y)}{(z-x)(w-y)}  
    \\
    &= \frac{1}{A_{\vec n,\vec k}(z,w)} \int_{T} \frac{A_{\vec n,\vec k}(z,w)-A_{\vec n,\vec k}(x,y) }{(z-x)(w-y)} \mathrm{d} \mu_j (x,y)  +\frac{1}{A_{\vec n,\vec k}(z,w)} \int_{T} \frac{A_{\vec n,\vec k}(x,y) }{(z-x)(w-y)} \mathrm{d} \mu_j (x,y).
\end{aligned}
$$
It follows that
\begin{equation}
    \label{eq:approximation-integrals}
    A_{\vec n,\vec k}(z,w)E_j(z,w) = \int_{T} \frac{A_{\vec n,\vec k}(z,w)-A_{\vec n,\vec k}(x,y) }{(z-x)(w-y)} \mathrm{d} \mu_j (x,y) +\int_{T} \frac{A_{\vec n,\vec k}(x,y) }{(z-x)(w-y)} \mathrm{d} \mu_j (x,y).
\end{equation}

In what follows, we determine $\Phi_{\vec n,\vec k}^{(j)}$ and $A_{\vec n,\vec k}$ such that this structural decomposition satisfies the asymptotic conditions required by \eqref{eq:approximation}.

\subsection{The denominator of the approximation}

We now focus on the second integral in \eqref{eq:approximation-integrals}. Our objective is to ensure that this term becomes $\mathcal{O}(z^{-(n_j-k_j)-1}w^{-k_j-1})$. Substituting the Laurent expansion of $E_j$ from \eqref{eq:Ej-series} yields
$$
\int_{T} \frac{A_{\vec n,\vec k}(x,y) }{(z-x)(w-y)} \mathrm{d} \mu_j (x,y) = \sum_{l=0}^{\infty} \sum_{m=0}^{\infty} \frac{1}{z^{l+1} w^{m+1}}\int_T x^l y^m A_{\vec n,\vec k}(x,y)\,\mathrm{d} \mu_j(x,y),
$$
which can be expressed as
\begin{equation}
    \label{eq:series-coefficients}
    \int_{T} \frac{A_{\vec n,\vec k}(x,y) }{(z-x)(w-y)} \mathrm{d} \mu_j (x,y) =\sum_{l=0}^{\infty} \sum_{m=0}^{\infty} \frac{b_{l,m}^{(j)}}{z^{l+1} w^{m+1}}, \quad \text{ where } b_{l,m}^{(j)}=\int_T x^l y^m A_{\vec n,\vec k}(x,y)\,\mathrm{d} \mu_j(x,y).
\end{equation}
To impose the condition $\displaystyle\sum_{l=0}^{\infty} \sum_{m=0}^{\infty} \frac{b_{l,m}^{(j)}}{z^{l+1} w^{m+1}}=\mathcal{O}(z^{-(n_j-k_j)-1}w^{-k_j-1})$, we require that \citep[Problem 1]{Sor02}

\begin{equation}
    \label{eq:HP-problem}
    b_{l,m}^{(j)} = 0, \quad \text{ for } (l,m)\in \mathbb L_{(n_j,k_j)},\quad j=1,2,
\end{equation}
where
$$
\mathbb L_{(n_j,k_j)} = \{(l,m) \in \mathbb N_0^2: l<n_j-k_j \text{ or } m<k_j\}.
$$
By recalling the integral representation of the coefficients $b_{l,m}^{(j)}$ in \eqref{eq:series-coefficients}, the conditions in \eqref{eq:HP-problem} can be rewritten as
\begin{equation}
    \label{eq:H-P-orthogonality}
    \int_T x^l y^m A_{\vec n,\vec k}(x,y)\,\mathrm{d} \mu_j(x,y) = 0\qquad \text{ if } 0\leq l < n_j-k_j \text{ or } 0\leq m< k_j,\quad j=1,2.
\end{equation}

Consequently, problem \eqref{eq:HP-problem} reduces to finding a polynomial $A_{\vec n,\vec k}$ of total degree at most $n_1+n_2$ that satisfies the orthogonality conditions \eqref{eq:H-P-orthogonality}. These conditions coincide with the orthogonality relations in \eqref{eq:orthogonality-r-measures} satisfied by the polynomial $U_{\vec n,\vec k}^{(\vec \alpha,\vec \beta,\gamma)}(x,y)$. The above discussion establishes the following result.

\begin{proposition}\label{prop:denominator}
    The polynomial $U_{\vec n,\vec k}^{(\vec \alpha,\vec \beta,\gamma)}$, introduced in \eqref{eq:pol-definition} is a solution to the problem \eqref{eq:approximation}.
\end{proposition}

As a consequence, the polynomial $U_{\vec n,\vec k}^{(\vec \alpha,\vec \beta,\gamma)}$ may be chosen as a common denominator of the approximant $R_{\vec n,\vec k}^{(j)}(z,w)$, enabling us to set $A_{\vec n,\vec k}=U_{\vec n,\vec k}^{(\vec \alpha,\vec \beta,\gamma)}$ in future computations (see Section~\ref{sec:simulations}).

\begin{remark}
    Although the polynomial $U_{\vec n,\vec k}^{(\vec \alpha,\vec \beta,\gamma)}$ satisfies the requirements imposed on the denominator of the approximants, it is not the only possible choice. Any polynomial fulfilling the orthogonality conditions \eqref{eq:H-P-orthogonality} provides a valid denominator for the approximation problem. Therefore, unlike in the univariate framework, the denominator is not uniquely determined by the approximation conditions alone.
\end{remark}

To contextualize these results, recall the univariate case described in Section~\ref{sec:preliminaries}, Equation~\eqref{eq:H-P-univariate}. At this stage, we have established an analogue of $P_{\vec n}(x)$, namely the bivariate Jacobi--Pi\~neiro polynomial $U_{\vec n,\vec k}^{(\vec \alpha,\vec \beta,\gamma)}(x,y)$. However, the construction is not yet complete, since we still need an analogue of the polynomial $Q_{\vec n,j}(x)$ appearing in the univariate Hermite--Padé problem. In the bivariate setting, we investigate the corresponding object and show that, although it is no longer a polynomial, it can still be computed explicitly.

\subsection{The numerator of the approximation}
Thus far, we have established that \eqref{eq:approximation-integrals} can be written as
\begin{equation}
    A_{\vec n,\vec k}(z,w)E_j(z,w) = \int_{T} \frac{A_{\vec n,\vec k}(z,w)-A_{\vec n,\vec k}(x,y) }{(z-x)(w-y)} \mathrm{d} \mu_j (x,y) + \mathcal{O}(z^{-n_j+k_j-1} w^{-k_j-1}),
\end{equation}
as $z,w\longrightarrow\infty$ for $j=1,2$, provided that $A_{\vec n,\vec k}$ satisfies the orthogonality conditions \eqref{eq:H-P-orthogonality}.

In this section, we investigate the integral
\begin{equation}
    \label{eq:integral-numerator}
    \Phi_{\vec n,\vec k}^{(j)}(z,w):=\int_{T} \frac{A_{\vec n,\vec k}(z,w)-A_{\vec n,\vec k}(x,y) }{(z-x)(w-y)} \, \mathrm{d} \mu_j (x,y),
\end{equation}
which defines the numerator of the approximant and serves as the bivariate analogue of $Q_{\vec n,j}(x)$ from the univariate setting. 

Let us expand the polynomial $A_{\vec n,\vec k}$ in terms of its coefficients as
\begin{equation}
    \label{eq:A-coefficients}
    A_{\vec n,\vec k}(x,y)=\sum_{p=0}^{n_1+n_2}\sum_{i=0}^p a_{p-i,i} x^{p-i} y^i.
\end{equation}

By adding and subtracting the intermediate term $A_{\vec n,\vec k}(x,w)$ within the integrand of \eqref{eq:integral-numerator}, we obtain the structural decomposition
\begin{equation}
    \label{eq:numerator-integrals}
    \Phi_{\vec n,\vec k}^{(j)}(z,w) =  \int_{T} \frac{A_{\vec n,\vec k}(z,w)-A_{\vec n,\vec k}(x,w)}{(z-x)(w-y)} \mathrm{d} \mu_j (x,y) + \int_{T} \frac{A_{\vec n,\vec k}(x,w)-A_{\vec n,\vec k}(x,y)}{(z-x)(w-y)} \mathrm{d} \mu_j (x,y).
\end{equation}

We analyze the first integral component, defined as
\begin{equation}
\label{eq:psi_1-1}
\begin{aligned}
    \psi_1^{(j)}(z,w):=&\int_{T} \frac{A_{\vec n,\vec k}(z,w)-A_{\vec n,\vec k}(x,w)}{(z-x)(w-y)} \mathrm{d} \mu_j (x,y)\\
    =&\int_{0}^{1}\dfrac{\int_0^{1-y}\frac{A_{\vec n,\vec k}(z,w)-A_{\vec n,\vec k}(x,w)}{z-x}x^{\alpha_j}y^{\beta_j}(1-x-y)^\gamma\mathrm{d} x}{w-y}\mathrm{d} y.
\end{aligned}
\end{equation}

By utilizing the coefficient notation $a_{p-i,i}$ for the polynomial $A_{\vec n,\vec k}(z,w)$ introduced in \eqref{eq:A-coefficients}, it follows that
\begin{equation*}
\frac{A_{\vec n,\vec k}(z,w) -A_{\vec n,\vec k}(x,w)}{z-x}= \sum_{p=0}^{n_1+n_2} 
\sum_{i=0}^{p} a_{p-i,i}  w^i \frac{z^{p-i}-x^{p-i} }{z-x}
= \sum_{p=1}^{n_1+n_2} 
\sum_{i=0}^{p-1} \sum_{h=0}^{p-i-1} a_{p-i,i}  w^i z^{p-i-1-h} x^h,
\end{equation*}
where the second equality is obtained by applying the algebraic identity
\begin{equation}
    \label{eq:polynomial-division}
    \frac{x^n-y^n}{x-y}=\begin{cases}
        0 & \text{ if } n=0,\\
        \displaystyle\sum_{k=0}^{n-1} x^{n-k-1}y^k=\sum_{k=0}^{n-1} x^ky^{n-k-1}& \text{ if } n\geq 1.
    \end{cases}
\end{equation}

Then, the integral in the numerator of \eqref{eq:psi_1-1} becomes

$$\begin{aligned}
&\int_0^{1-y} \frac{A_{\vec n,\vec k}(z,w) -A_{\vec n,\vec k}(x,w)}{z-x} x^{\alpha_j} y^{\beta_j} (1-x-y)^{\gamma} \, \dr x
\\ &=\sum_{p=1}^{n_1+n_2} 
\sum_{i=0}^{p-1} \sum_{h=0}^{p-i-1} a_{p-i,i}  w^i z^{p-i-1-h} y^{\beta_j} \int_0^{1-y} x^{\alpha_j+h}  (1-x-y)^{\gamma} \, \dr x\\
&= \sum_{p=1}^{n_1+n_2} 
\sum_{i=0}^{p-1} \sum_{h=0}^{p-i-1} a_{p-i,i}  w^i z^{p-i-1-h} y^{\beta_j} (1-y)^{\alpha_j+h+\gamma+1}\frac{\Gamma(\alpha_j+h+1)\Gamma(\gamma+1) }{\Gamma(\alpha_j+h+\gamma+2)}\\
&= \left(\sum_{p=1}^{n_1+n_2} 
\sum_{i=0}^{p-1} \sum_{h=0}^{p-i-1}\frac{\Gamma(\alpha_j+h+1) }{\Gamma(\alpha_j+h+\gamma+2)} a_{p-i,i}  w^i z^{p-i-1-h} (1-y)^{h}
\right)\left(  \Gamma(\gamma+1) y^{\beta_j}(1-y)^{\alpha_j+\gamma+1} \right)\\
&= \tilde P^{(\alpha_j,\gamma)} (z,w,y) \omega^{(\alpha_j,\beta_j,\gamma)}(y),
\end{aligned}$$
where we denote
\begin{equation}
\label{eq:numerator-1}
    \begin{array}{rl}
        \tilde P^{(\alpha,\gamma)} (z,w,y) &:=\displaystyle\sum_{p=1}^{n_1+n_2} 
        \sum_{i=0}^{p-1} \sum_{h=0}^{p-i-1}\frac{\Gamma(\alpha+h+1) }{\Gamma(\alpha+h+\gamma+2)} a_{p-i,i}  w^i z^{p-i-1-h} (1-y)^{h}
        , \\[12pt] 
        \omega^{(\alpha,\beta,\gamma)}(y)&:=\Gamma(\gamma+1) y^{\beta}(1-y)^{\alpha+\gamma+1} \\[10pt]
        \dr \mu^{(\alpha,\beta,\gamma)}(y)&:= \omega^{(\alpha,\beta,\gamma)}(y)\dr y.
    \end{array}
\end{equation}

With these notations, we get that
\begin{equation}
    \label{eq:psi_1_integrals}
    \begin{aligned}
    \psi_1^{(j)}(z,w)&= \int_0^1 \dfrac{\tilde{P}^{(\alpha_j,\gamma)}(z,w,y)}{w-y}\dr \mu^{(\alpha_j,\beta_j,\gamma)}(y)\\
    &= \int_0^1 \dfrac{\tilde{P}^{(\alpha_j,\gamma)}(z,w,y)-\tilde{P}^{(\alpha_j,\gamma)}(z,w,w)}{w-y}\dr \mu^{(\alpha_j, \beta_j, \gamma)}(y)\\ &+\int_0^1 \dfrac{\tilde{P}^{(\alpha_j,\gamma)}(z,w,w)}{w-y}\dr \mu^{(\alpha_j,\beta_j,\gamma)}(y)
\end{aligned}
\end{equation}
Denote $\psi_{1,1}^{(j)}$ and $\psi_{1,2}^{(j)}$ the first and second integrals of \eqref{eq:psi_1_integrals}, respectively.

Now focus on the second integral of \eqref{eq:numerator-integrals}
\begin{equation*}
    \begin{aligned}
    \psi^{(j)}_{2}(z,w)&:=\int_{T} \frac{A_{\vec n,\vec k}(x,w)-A_{\vec n,\vec k}(x,y)}{(z-x)(w-y)} \mathrm{d} \mu_j (x,y)\\
    &=  \int_{0}^{1}\dfrac{\int_{0}^{1-x} \frac{A_{\vec n,\vec k}(x,w)-A_{\vec n,\vec k}(x,y)}{w-y} x^{\alpha_j}y^{\beta_j} (1-x-y)^\gamma \,\dr y }{z-x}\, \dr x
    \end{aligned}
\end{equation*}
Using the expansion \eqref{eq:A-coefficients} and proceeding as above, we get
\begin{equation}
    \label{eq:psi_2_integrals}
    \begin{aligned}
    \psi^{(j)}_2(z,w)&=\int_0^1 \dfrac{\tilde P^{(\beta_j,\gamma)}(w,x,x)}{z-x}\dr \mu^{(\beta_j,\alpha_j,\gamma)}(x)\\
    & = \int_0^1 \dfrac{\tilde P^{(\beta_j,\gamma)}(w,x,x)-\tilde P^{(\beta_j,\gamma)}(w,z,z)}{z-x}\dr \mu^{(\beta_j,\alpha_j,\gamma)}(x)\\ &+\int_0^1 \dfrac{\tilde P^{(\beta_j,\gamma)}(w,z,z)}{z-x}\dr \mu^{(\beta_j,\alpha_j,\gamma)}(x)
    \end{aligned}
\end{equation}
Denote the first and second integrals of \eqref{eq:psi_2_integrals} as $\psi_{2,1}^{(j)}$ and $\psi_{2,2}^{(j)}$, respectively.

We first focus on $\psi_{1,2}^{(j)}$ and $\psi_{2,2}^{(j)}$, the second integrals in \eqref{eq:psi_1_integrals} and \eqref{eq:psi_2_integrals}. We will show that they can be expressed in terms of hypergeometric functions. First, we have that

\begin{equation*}
\begin{split}
\psi_{1,2}^{(j)}(z,w)&=\int_0^1 \frac{\tilde P^{(\alpha_j,\gamma)}(z,w,w) }{w-y}\, \dr \mu^{(\alpha_j,\beta_j,\gamma)}(y)\\ & = \Gamma(\gamma+1) \tilde P^{(\alpha_j,\gamma)}(z,w,w) \int_0^1 \frac{1 }{w-y} y^{\beta_j}(1-y)^{\alpha_j+\gamma+1} \, \dr y\\
&= \Gamma(\gamma+1) \tilde P^{(\alpha_j,\gamma)}(z,w,w) \sum_{k=0}^{\infty} \frac{\int_0^1 y^{\beta_j+k}(1-y)^{\alpha_j+\gamma+1} \, \dr y}{w^{k+1}}\\ &=\Gamma(\gamma+1) \tilde P^{(\alpha_j,\gamma)}(z,w,w) \sum_{k=0}^{\infty} 
\frac{\frac{ \Gamma(\beta_j+k+1)\Gamma(\alpha_j+\gamma+2) }{\Gamma(\alpha_j+\beta_j+\gamma+k+3)}}{w^{k+1}}\\
&= \Gamma(\gamma+1)\Gamma(\alpha_j+\gamma+2) \tilde P^{(\alpha_j,\gamma)}(z,w,w)\sum_{k=0}^{\infty}\frac{\Gamma(\beta_j+k+1)}{\Gamma(\alpha_j+\beta_j+\gamma+k+3)w^{k+1}},
\end{split}
\end{equation*}
for sufficiently large $w$. Recall the hypergeometric function
$$
{_2}F_1\left( \begin{matrix} a,\, b \\ c \end{matrix} \,;\, z \right) = \sum_{k=0}^{\infty}\dfrac{(a)_k (b)_k}{(c)_k}\dfrac{z^k}{k!} = \sum_{k=0}^{\infty}\dfrac{\Gamma(a+k)\Gamma(b+k)\Gamma(c)z^k}{\Gamma(a)\Gamma(b)\Gamma(c+k)k!}.
$$
Defining
\begin{equation}
    \label{eq:phi_hyp}
    \phi^{(\alpha,\beta,\gamma)}_{\mathrm{hyp}}(z,w):= \dfrac{\Gamma(\alpha+\gamma+2)\Gamma(\beta+1)\Gamma(\gamma+1)}{\Gamma(\alpha+\beta+\gamma+3) w} \tilde P^{(\alpha,\gamma)}(z,w,w)\, {_2}F_1\left( \begin{matrix} \beta+1,1 \\ \alpha+\beta+\gamma+3 \end{matrix} \,;\, \dfrac 1 w \right),
\end{equation}
the final expression of $\psi_{1,2}^{(j)}$ is
\begin{equation}
\label{eq:psi_12}
    \psi_{1,2}^{(j)}(z,w) = \int_0^1 \frac{\tilde P^{(\alpha_j,\gamma)}(z,w,w) }{w-y}\, \dr \mu^{(\alpha_j,\beta_j,\gamma)}(y)  = \phi^{(\alpha_j,\beta_j,\gamma)}_{\mathrm{hyp}}(z,w).
\end{equation}

Now, we observe that the integral $\psi_{2,2}^{(j)}$ is the same as $\psi_{1,2}^{(j)}$ but with the roles of $z$ and $w$ interchanged, and with the parameters $\alpha_j$ and $\beta_j$ swapped. Therefore, we have
\begin{equation}
\label{eq:psi_22}
\psi_{2,2}^{(j)}(z,w) = \int_0^1 \frac{\tilde P^{(\beta_j,\gamma)}(w,z,z) }{z-x}\, \dr \mu^{(\beta_j,\alpha_j,\gamma)}(x)  = \phi^{(\beta_j,\alpha_j,\gamma)}_{\mathrm{hyp}}(w,z).
\end{equation}

Now, we will focus on $\psi_{1,1}^{(j)}$ and $\psi_{2,1}^{(j)}$, the first integrals in \eqref{eq:psi_1_integrals} and \eqref{eq:psi_2_integrals}. We will show that they are both polynomials.

On the one hand, we have
$$
\psi_{1,1}^{(j)}(z,w)=\int_0^1 \dfrac{\tilde{P}^{(\alpha_j,\gamma)}(z,w,y)-\tilde{P}^{(\alpha_j,\gamma)}(z,w,w)}{w-y}\dr \mu^{(\alpha_j,\beta_j,\gamma)}(y)
$$

See \eqref{eq:numerator-1} and observe that $\tilde P^{(\alpha,\gamma)} (z,w,y)$ is a polynomial of degree $n_1+n_2-1$ in three variables $z,w,y$. In order to simplify some notations, we denote $\tilde P^{(\alpha,\gamma)} (z,w,y)$ and its coefficients as
\begin{equation}
    \label{eq:pol-P-tilde}
    \tilde P^{(\alpha,\gamma)} (z,w,y) = \sum_{p=0}^{n_1+n_2-1}\sum_{i=0}^p\sum_{h=0}^{p-i}\tilde a_{p-i-h,i,h}^{(\alpha,\gamma)}z^{p-i-h}w^i y^h.
\end{equation}

By this expansion, the expression of the measure $\mu^{(\alpha_j,\beta_j,\gamma)}$ in \eqref{eq:numerator-1}, and \eqref{eq:polynomial-division}, we get
$$
\psi_{1,1}^{(j)}(z,w)=-\Gamma(\gamma+1)\sum_{p=0}^{n_1+n_2-1}\sum_{i=0}^p\sum_{h=1}^{p-i}\sum_{k=0}^{h-1}\tilde a_{p-i-h,i,h}^{(\alpha_j,\gamma)}z^{p-i-h}w^{i+h-k-1}\int_0^1 y^{k+\beta_j}(1-y)^{\alpha_j+\gamma+1}\dr y,
$$
so that

\begin{multline}
 \label{eq:psi_11}
\psi_{1,1}^{(j)}(z,w)= -\Gamma(\gamma+1)\Gamma(\alpha_j+\gamma+2) \\ \times \sum_{p=0}^{n_1+n_2-1}\sum_{i=0}^p\sum_{h=1}^{p-i}\sum_{k=0}^{h-1}\dfrac{\Gamma(\beta_j+k+1)}{\Gamma(\alpha_j+\beta_j+\gamma+k+3)}\tilde a_{p-i-h,i,h}^{(\alpha_j,\gamma)}z^{p-i-h}w^{i+h-k-1},
\end{multline}
which is also a polynomial.

On the other hand, the first integral of $\psi^{(j)}_2$
$$
\psi_{2,1}^{(j)}(z,w)=\int_0^1 \dfrac{\tilde P^{(\beta_j,\gamma)}(w,x,x)-\tilde P^{(\beta_j,\gamma)}(w,z,z)}{z-x}\dr \mu^{(\beta_j,\alpha_j,\gamma)}(x),
$$
which is slightly different from the previous one. Using \eqref{eq:pol-P-tilde}, the expression of $\mu^{(\beta_j,\alpha_j,\gamma)}$ and \eqref{eq:polynomial-division}, it might be written as
$$
\psi_{2,1}^{(j)}(z,w)=-\Gamma(\gamma+1)\sum_{p=0}^{n_1+n_2-1}\sum_{i=0}^p\sum_{h=1}^{p-i}\sum_{k=0}^{h-1}\tilde a_{p-i-h,i,h}^{(\beta_j,\gamma)}w^{p-i-h}z^{i+h-k-1}\int_0^1 x^{k+\alpha_j}(1-x)^{\beta_j+\gamma+1}\dr x,
$$
thus,

\begin{multline}
\label{eq:psi_21}
   \psi_{2,1}^{(j)}(z,w)=-\Gamma(\gamma+1)\Gamma(\beta_j+\gamma+2)\\ \times\sum_{p=0}^{n_1+n_2-1}\sum_{i=0}^p\sum_{h=1}^{p-i}\sum_{k=0}^{i+h-1}\dfrac{\Gamma(\alpha_j+k+1)}{\Gamma(\alpha_j+\beta_j+\gamma+k+3)}\tilde a_{p-i-h,i,h}^{(\beta_j,\gamma)}w^{p-i-h}z^{i+h-k-1}.
\end{multline}

Observe that \eqref{eq:psi_11} and \eqref{eq:psi_21} are almost reciprocal, up to a difference concerning the last sum, whose upper limit is $h-1$ in the case of \eqref{eq:psi_11}, while it is $i+h-1$ in \eqref{eq:psi_21}.

In summary, we have proved the following result about $\Phi^{(j)}(z,w)$.

\begin{proposition}\label{prop:numerator}
    Let $A_{\vec n,\vec k}(z,w)$ be a solution of the problem \eqref{eq:approximation}. Then the term $\Phi^{(j)}(z,w)$ defined in \eqref{eq:integral-numerator} has the expression

    $$
    \Phi^{(j)}_{\vec n,\vec k}(z,w)=\Phi^{(j)}_{\mathrm{pol}}(z,w)+\Phi^{(j)}_{\mathrm{hyp}}(z,w)
    $$
    where
    \begin{itemize}
        \item $
    \Phi^{(j)}_{\mathrm{pol}}(z,w)=\psi_{1,1}^{(j)}(z,w)+\psi_{2,1}^{(j)}(z,w)
    $ with $\psi_{1,1}^{(j)}$ and $\psi_{2,1}^{(j)}$ given in \eqref{eq:psi_11} and \eqref{eq:psi_21} respectively is a polynomial, and 
    \item  $\Phi^{(j)}_{\mathrm{hyp}}(z,w)=\psi_{1,2}^{(j)}(z,w)+\psi_{2,2}^{(j)}(z,w)$, where $\psi_{1,2}^{(j)}$ and $\psi_{2,2}^{(j)}$ are given in \eqref{eq:psi_12} and \eqref{eq:psi_22} respectively, being both products of hypergeometric functions by a polynomial.   
    \end{itemize}
\end{proposition}

\begin{remark} 
    Instead of expressing $\psi_{1,2}^{(j)}(z,w)$ and $\psi_{2,2}^{(j)}(z,w)$ in terms of the hypergeometric function ${}_2F_1$, one may approximate the corresponding Stieltjes transforms
    \[
    \int_0^1 \frac{\mathrm{d}\mu^{(\alpha_j,\beta_j,\gamma)}(y)}{w-y}
    \qquad\text{and}\qquad
    \int_0^1 \frac{\mathrm{d}\mu^{(\beta_j,\alpha_j,\gamma)}(x)}{z-x}
    \]
    by means of univariate Pad\'e approximants. Taking into account the explicit expressions of the measure $\mu^{(\alpha,\beta,\gamma)}$ in \eqref{eq:numerator-1} and following the methodology described in \citep{VA06}, it follows that the denominators of these Pad\'e approximants correspond, up to a linear change of variables, to Jacobi polynomials.
\end{remark}

Combining Propositions~\ref{prop:denominator} and \ref{prop:numerator}, we conclude that \eqref{eq:approximation} holds with
\[
A_{\vec n,\vec k}(z,w)=
U_{\vec n,\vec k}^{(\vec\alpha,\vec\beta,\gamma)}(z,w),
\]
where the polynomial
$U_{\vec n,\vec k}^{(\vec\alpha,\vec\beta,\gamma)}$
is defined in \eqref{eq:pol-definition}. Therefore, the quotient
\begin{equation}
    \label{eq:approx-R-j}
    R_{\vec n,\vec k}^{(j)}(z,w)=\frac{\Phi_{\vec n,\vec k}^{(j)}(z,w)}
{U_{\vec n,\vec k}^{(\vec\alpha,\vec\beta,\gamma)}(z,w)}
\end{equation}
satisfies the proposed Hermite--Padé-type approximation conditions for
$E_j(z,w)$.

To conclude this section, note that we have constructed approximations for two distinct functions, $E_1$ and $E_2$, sharing a common denominator: the bivariate multiple orthogonal polynomial $U_{\vec n, \vec k}^{(\vec \alpha, \vec \beta, \gamma)}(z,w)$. This structural unity constitutes the primary objective of Hermite--Pad\'e approximation schemes \citep{VA06}. In the subsequent section, we present numerical experiments validating these approximations for $E_1$ and $E_2$, as defined in \eqref{eq:Ej-integral}, using specific parameter configurations.

\section{Numerical Simulations}\label{sec:simulations}

In this section, we investigate the numerical performance of the Hermite--Padé-type approximants introduced in Section~\ref{sec:HP}. Specifically, we select concrete parameter values, construct the approximations for $E_1$ and $E_2$ following the procedure detailed in Section~\ref{sec:HP}, and assess their accuracy in approximating the target functions. Throughout this numerical study, we fix the parameters as $\alpha_1=0$, $\alpha_2=\frac{3}{2}$, $\beta_1=\frac{1}{2}$, $\beta_2=\frac{4}{3}$, and $\gamma=0$, so that the target functions to be approximated become
\begin{align*}
    E_1(z,w)&=\int_T\frac{\sqrt{y}}{(w-y) (z-x)}\mathrm{d} x\, \mathrm{d} y, & E_2(z,w)&=\int_T \frac{x^{3/2} y^{4/3}}{(w-y) (z-x)}\mathrm{d} x\, \mathrm{d} y.
\end{align*}

We consider multi-indices of the form $\vec k = (k,k)$ and $\vec n =(2k,2k)$, which reduces the asymptotic order conditions in \eqref{eq:approximation} to
$$
U_{\vec n,\vec k}^{(\vec\alpha,\vec\beta,\gamma)}(z,w) E_j(z,w) = \Phi_{\vec n,\vec k}^{(j)}(z,w) + \mathcal O (z^{-k-1} w^{-k-1}),\quad z,w\longrightarrow\infty, \quad j=1,2.
$$
For each pair of multi-indices $\vec k$ and $\vec n$, the functions \eqref{eq:approx-R-j} are explicitly computed.

As expected, as $k$ increases, both the degree of the denominator $U_{\vec n,\vec k}^{(\vec\alpha,\vec\beta,\gamma)}$ and the computational complexity of the approximation increase concurrently. Following the construction of these approximants, we evaluate their numerical accuracy.

We evaluate the performance of the approximation using three distinct validation sets, $\mathcal{C}_1$, $\mathcal{C}_2$, and $\mathcal{C}=\mathcal{C}_1\cup\mathcal{C}_2$, where
\begin{itemize}
    \item $\mathcal{C}_1$ consists of $N_1$ points randomly distributed in the region $([1,20]\times[1,20])\backslash ([10,20]\times[10,20])$,
    \item $\mathcal{C}_2$ consists of $N_2$ points randomly distributed in the region $[10,20]\times[10,20]$, and
    \item $\mathcal{C}=\mathcal{C}_1\cup\mathcal{C}_2$, represents the disjoint union of $\mathcal{C}_1$ and $\mathcal{C}_2$.

\end{itemize}

Figure~\ref{fig:test-points} illustrates the spatial distribution of $\mathcal{C}_1$ and $\mathcal{C}_2$ for $N_1=300$ and $N_2 = 100$. These specific domains are selected because the accuracy of the approximant $R_{\vec n,\vec k}^{(j)}$ inherently improves as $z,w\rightarrow\infty$. 

\begin{figure}
    \centering
    \includegraphics[width=0.4\linewidth]{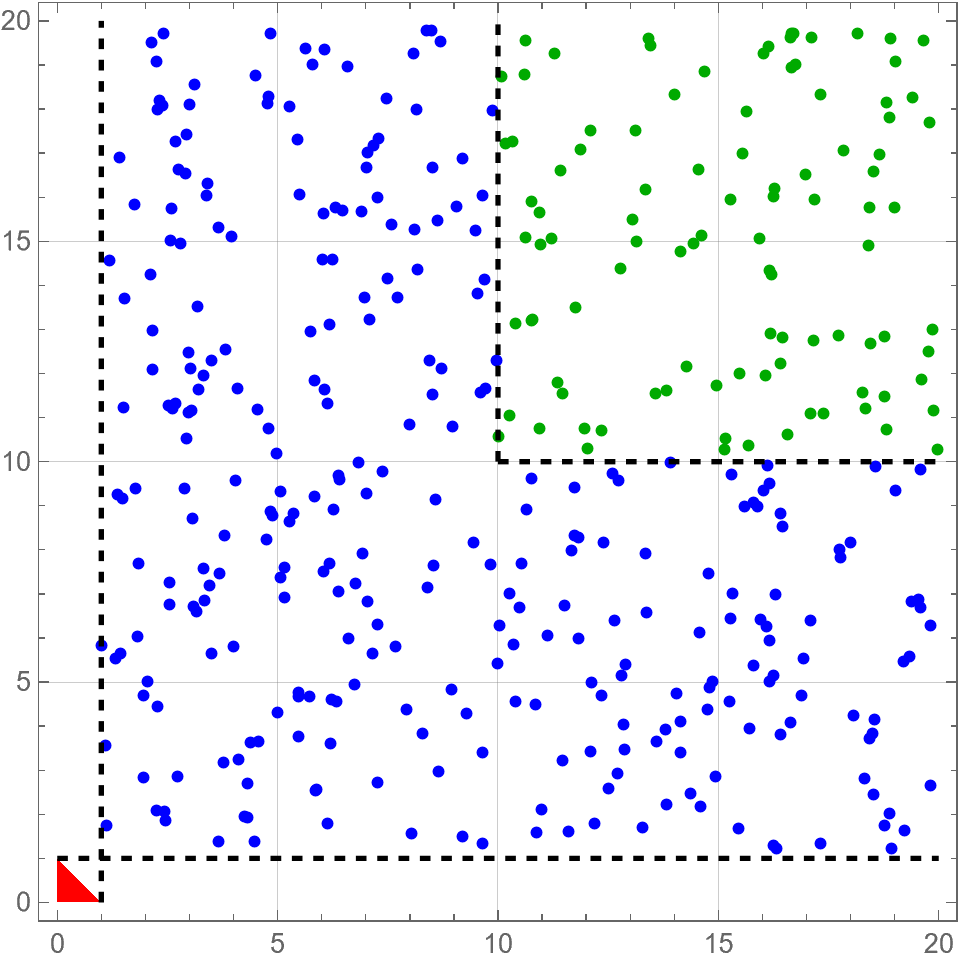}
    \caption{Spatial distribution of the test datasets. The blue markers represent elements of $\mathcal{C}_1$, while the green markers designate elements of $\mathcal{C}_2$.}
    \label{fig:test-points}
\end{figure}

For each configuration of $\vec n$ and $\vec k$, we assess the performance across the three test datasets. Specifically, for every target domain $\mathcal{T} \in \{\mathcal{C}_1, \mathcal{C}_2, \mathcal{C}\}$ and each index $j = 1,2$, we define the following error metrics.

\begin{itemize}
    \item \textbf{Mean Squared Error (MSE):} The mean of the squared local errors, defined as
    $$
    \mathrm{MSE}_j = \frac{1}{\#\mathcal{T}} \sum_{(x,y) \in \mathcal{T}} \left|E_j(x,y) - R_{\vec n,\vec k}^{(j)}(x,y)\right|^2. 
    $$
    \item \textbf{Maximum Absolute Error (MaxAE):} The maximum absolute local error, defined as
    $$
    \mathrm{MaxAE}_j = \max_{(x,y) \in \mathcal{T}} \left\{ \left| E_j(x,y) - R_{\vec n,\vec k}^{(j)}(x,y) \right| \right\}.
    $$
    \item \textbf{Mean Relative Error (MRE):} The average of the relative local errors, defined as
    $$
    \mathrm{MRE}_j = \frac{1}{\#\mathcal{T}} \sum_{(x,y) \in \mathcal{T}} \frac{\left| E_j(x,y) - R_{\vec n,\vec k}^{(j)}(x,y) \right|}{\left| E_j(x,y) \right|}.
    $$
    \item \textbf{Maximum Relative Error (MaxRE):} The maximum relative local error, defined as
    $$
    \mathrm{MaxRE}_j = \max_{(x,y) \in \mathcal{T}} \left\{ \frac{\left| E_j(x,y) - R_{\vec n,\vec k}^{(j)}(x,y) \right|}{\left| E_j(x,y) \right|} \right\}.
    $$

\end{itemize}

First, we choose $n_1 = n_2 = 2$ and $k_1=k_2=1$, so that the polynomial \eqref{eq:pol-definition} has degree $n_1+n_2=4$. Indeed, we have
\begin{multline}
    U_{\vec n,\vec k}^{\vec\alpha,\vec \beta,\gamma}(z,w)=U^{((0,3/2),(1/2,4/3),0)}_{(2,2),(1,1)}(z,w)=\frac{1045}{12}w^4 +672 w^3 z+\frac{2457 }{2}w^2 z^2+\frac{2200 }{3}w z^3+\frac{455 }{4}z^4\\-240 w^3-1274 w^2 z-1350 w z^2-308 z^3+\frac{455 }{2}w^2+700 w z+\frac{567 }{2}z^2-\frac{250 }{3}w-98 z+\frac{35}{4}.
\end{multline}

With these values, we use Proposition~\ref{prop:numerator} to compute $\Phi^{(1)}(z,w)$ and $\Phi^{(2)}(z,w)$, and \eqref{eq:approximation} reads: 
\begin{equation*}
    \begin{array}{c}
         U_{\vec n,\vec k}^{\vec\alpha,\vec \beta,\gamma}(z,w) E_1(z,w) = \Phi^{(1)}(z,w) + \mathcal{O}(z^{-2}w^{-2}), \\[5pt]
         
         U_{\vec n,\vec k}^{\vec\alpha,\vec \beta,\gamma}(z,w) E_2(z,w) = \Phi^{(2)}(z,w) + \mathcal{O}(z^{-2}w^{-2}). \\
    \end{array}
\end{equation*}

The results of the tests performed over the sets $\mathcal C_1,\mathcal C_2$ and $\mathcal C$ are displayed on Tables~\ref{tab:E1_2} and \ref{tab:E2_2}. Moreover, in this case we plot the functions $E_j(z,w)$, their computed approximation $R_{\vec n,\vec k}^{(j)}(z,w)$ and the absolute error $|E_j(z,w)-R_{\vec n,\vec k}^{(j)}(z,w)|$, as well as the relative error $\frac{|E_j(z,w)-R_{\vec n,\vec k}^{(j)}(z,w)|}{E_j(z,w)}$, see Figure~\ref{fig:simulations}.

\begin{table}[h]
\begin{tabular}{ccccc}
\hline
\begin{tabular}[c]{@{}c@{}}$\vec n =(2,2)$, \\ $\vec k=(1,1)$\end{tabular} & \textbf{MSE}$_1$ & \textbf{MaxAE}$_1$ & \textbf{MRE}$_1$ & \textbf{MaxRE}$_1$ \\ \hline\hline
$\mathcal C_1$                                                             & 2.60586e-6       & 0.01521          & 0.03599        & 0.26269           \\ \hline
$\mathcal C_2$                                                             & 3.7973e-10      & 4.5733e-5       & 0.01188        & 0.03197          \\ \hline
$\mathcal C$                                                               & 1.95449e-6      & 0.01521           & 0.02996         & 0.26268            \\ \hline
\end{tabular}
\caption{Error metrics on the approximation of $E_1$ using $\vec n =(2,2), \vec k = (1,1)$.}
\label{tab:E1_2}
\end{table}

\begin{table}[h]
\begin{tabular}{ccccc}
\hline
\begin{tabular}[c]{@{}c@{}}$\vec n =(2,2)$, \\ $\vec k=(1,1)$\end{tabular} & \textbf{MSE}$_2$ & \textbf{MaxAE}$_2$ & \textbf{MRE}$_2$ & \textbf{MaxRE}$_2$ \\ \hline\hline
$\mathcal C_1$                                                             & 7.15416e-8       & 0.00309         & 0.06765        & 0.47147           \\ \hline
$\mathcal C_2$                                                             & 5.6234e-12      & 5.5491e-6         & 0.02179       & 0.05779          \\ \hline
$\mathcal C$                                                               & 5.36576e-8       & 0.00309         & 0.05618        & 0.47147           \\ \hline
\end{tabular}
\caption{Error metrics on the approximation of $E_2$ using $\vec n =(2,2), \vec k = (1,1)$.}
\label{tab:E2_2}
\end{table}

\begin{figure}[ht]
    \centering\begin{tabular}{cc}
    \includegraphics[height=6cm]{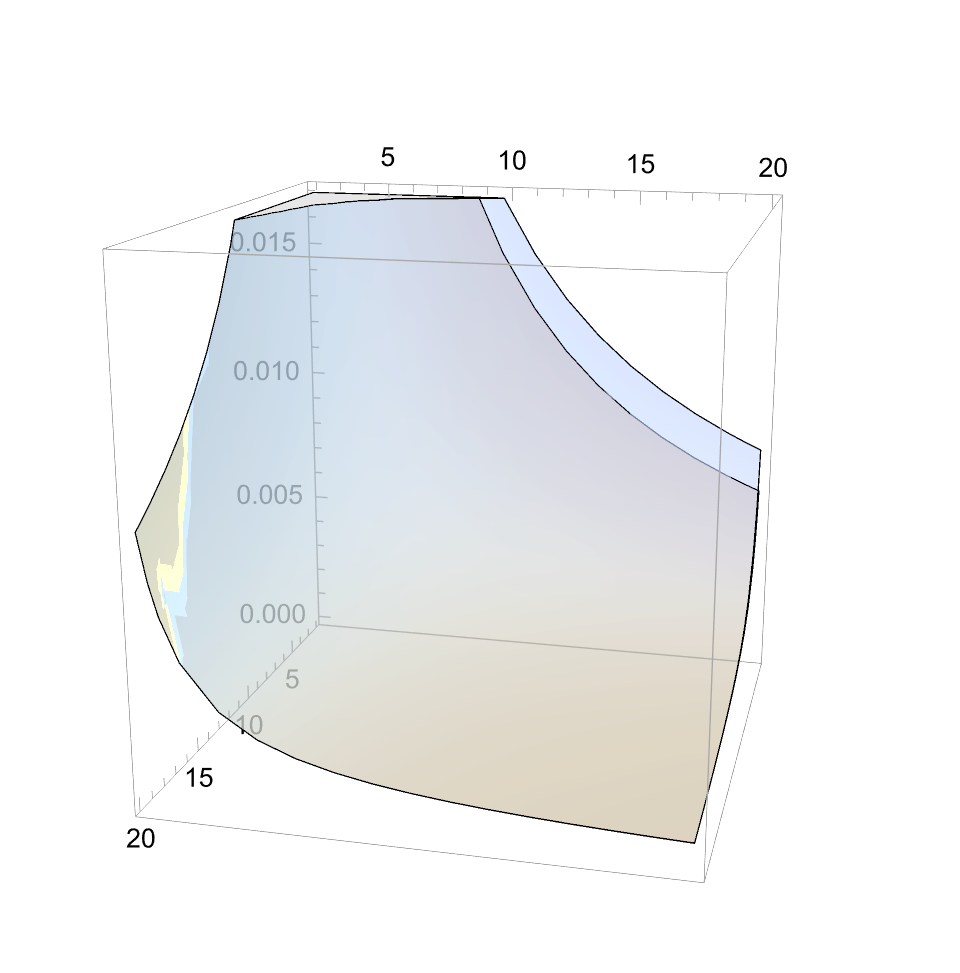}     & 
    \includegraphics[height=6cm]{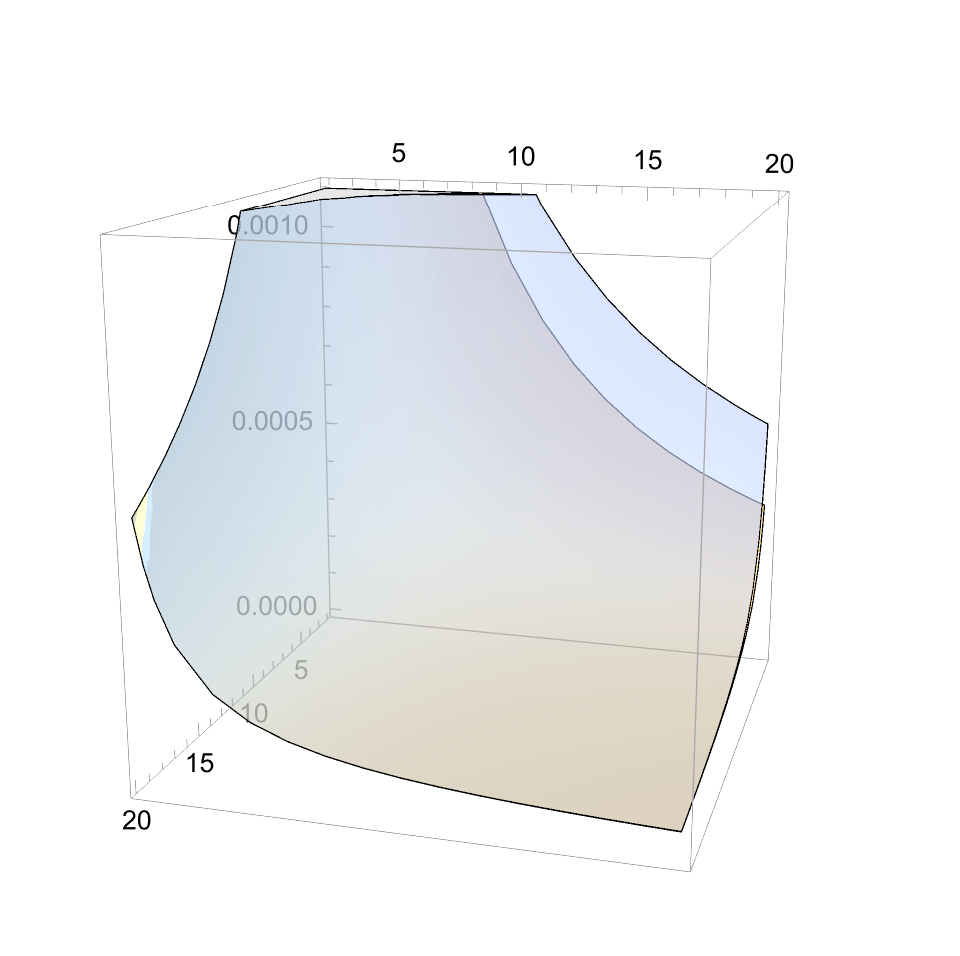}  \\
    (a) $E_1$ and $R_{\vec n,\vec k}^{(1)}$     &  (b) $E_2$ and $R_{\vec n,\vec k}^{(2)}$ \\
    
    \includegraphics[width=0.45\textwidth]{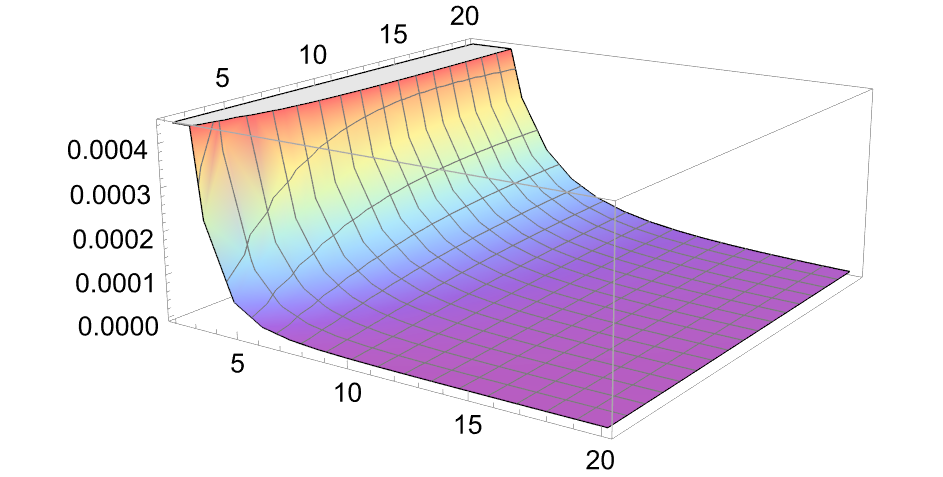}     & 
    \includegraphics[width=0.45\textwidth]{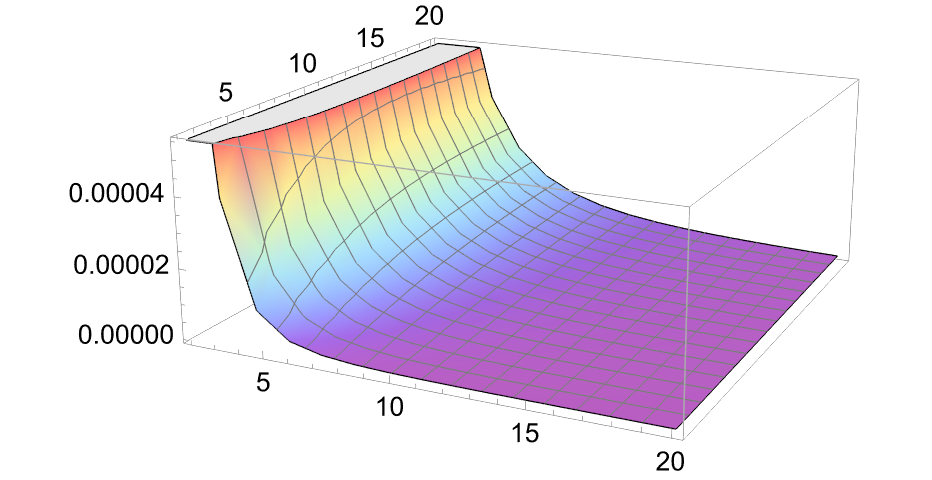}  \\
    (c)  $|E_1-R_{\vec n,\vec k}^{(1)}|$   &  (d) $|E_2-R_{\vec n,\vec k}^{(2)}|$ \\

    \includegraphics[width=0.45\textwidth]{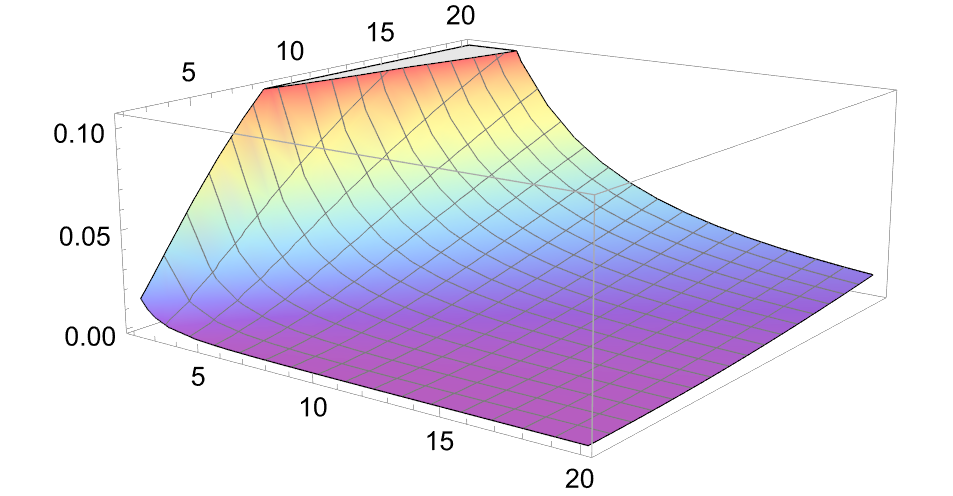}     & 
    \includegraphics[width=0.45\textwidth]{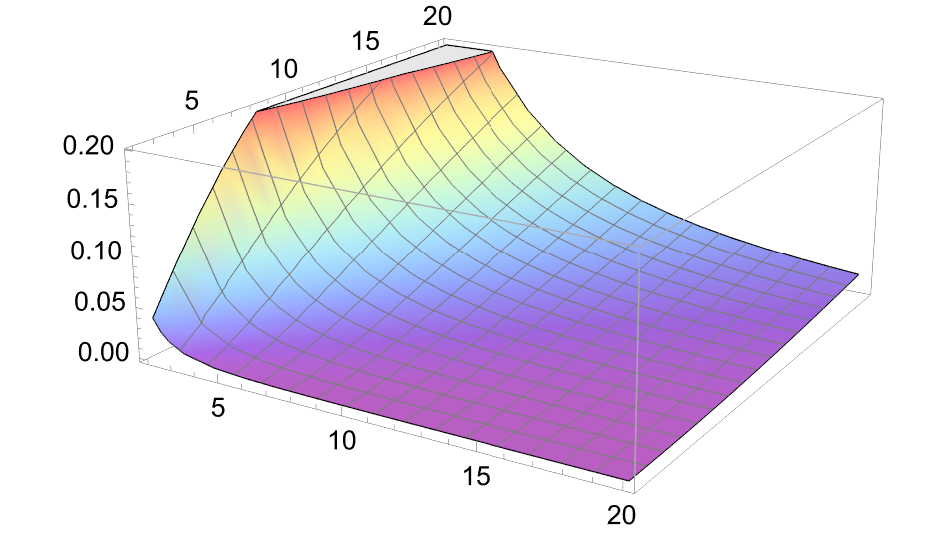}  \\
    (e)  $\frac{|E_1-R_{\vec n,\vec k}^{(1)}|}{|E_1|}$   &  (f) $\frac{|E_2-R_{\vec n,\vec k}^{(2)}|}{|E_2|}$ \\
         
    \end{tabular}
    \caption{First row: plot of $E_j$ (blue) and the approximation $R_{\vec n,\vec k}^{(j)}$ (yellow). Second row: plot of the absolute error $|E_j-R_{\vec n,\vec k}^{(j)}|$. Third row: plot of the relative error $\frac{|E_j-R_{\vec n,\vec k}^{(j)}|}{|E_j|}$.}
    \label{fig:simulations}
\end{figure}

Observing the plots in Figure~\ref{fig:simulations} and Tables~\ref{tab:E1_2} and \ref{tab:E2_2}, we notice that the errors on $\mathcal C_2$ are always smaller than those on $\mathcal C_1$. This behavior is consistent with the asymptotic nature of the approximation, since the remainder term is controlled as $z,w\to\infty$. Consequently, the approximation becomes more accurate on $\mathcal C_2$, whose points are located farther from the origin (see Figure~\ref{fig:test-points}). Furthermore, Figures~\ref{fig:simulations}(c)--(f) show that both the absolute and relative errors decrease as $(z,w)$ moves away from the origin.

Finally, we set $k=2$, so that $\vec n =(4,4)$, $\vec k =(2,2)$. In this case $U_{\vec n,\vec k}^{\vec\alpha,\vec \beta,\gamma}=U_{(4,4),(2,2)}^{((0,3/2),(1/2,4/3),0)}$, the common denominator of $R_{\vec n,\vec k}^{(j)}$ is a polynomial of degree $n_1+n_2 = 8$, whose expression has been omitted for brevity. Then, in this case the function $\Phi^{(j)}$ and the polynomial $U_{\vec n,\vec k}^{\vec\alpha,\vec \beta,\gamma}$ satisfy
\begin{equation*}
    \begin{array}{c}
         U_{\vec n,\vec k}^{\vec\alpha,\vec \beta,\gamma}(z,w) E_1(z,w) = \Phi^{(1)}(z,w) + \mathcal{O}(z^{-3}w^{-3}), \\[5pt]
         
         U_{\vec n,\vec k}^{\vec\alpha,\vec \beta,\gamma}(z,w) E_2(z,w) = \Phi^{(2)}(z,w) + \mathcal{O}(z^{-3}w^{-3}). \\
    \end{array}
\end{equation*}

The results of the tests are shown in Tables~\ref{tab:E1_4} and \ref{tab:E2_4}. Again, as expected, the results on $\mathcal C_2$ are the lowest ones.

\begin{table}[h]
\begin{tabular}{ccccc}
\hline
\begin{tabular}[c]{@{}c@{}}$\vec n =(4,4)$, \\ $\vec k=(2,2)$\end{tabular} & \textbf{MSE}$_1$ & \textbf{MaxAE}$_1$ & \textbf{MRE}$_1$ & \textbf{MaxRE}$_1$ \\ \hline\hline 
$\mathcal C_1$                                                             & 3.70322e-7      & 0.00412       & 0.00841        & 0.17016          \\ \hline
$\mathcal C_2$                                                             & 4.0023e-13       & 2.10954e-6          & 0.00029        & 0.00155           \\ \hline

$\mathcal C$                                                               & 2.77741e-7      & 0.00412           & 0.00638         & 0.17016            \\ \hline
\end{tabular}
\caption{Error metrics on the approximation of $E_1$ using $\vec n =(4,4), \vec k = (2,2)$.}
\label{tab:E1_4}
\end{table}

\begin{table}[h]
\begin{tabular}{ccccc}
\hline
\begin{tabular}[c]{@{}c@{}}$\vec n =(4,4)$, \\ $\vec k=(2,2)$\end{tabular} & \textbf{MSE}$_2$ & \textbf{MaxAE}$_2$ & \textbf{MRE}$_2$ & \textbf{MaxRE}$_2$ \\ \hline\hline
$\mathcal C_1$                                                             & 1.08238e-8       & 0.00082         & 0.01782        & 0.33366           \\ \hline
$\mathcal C_2$                                                             & 7.5652e-15      & 2.85235e-7         & 0.00060       & 0.00313          \\ \hline
$\mathcal C$                                                               & 8.11786e-9       & 0.00082         & 0.01352        & 0.33366           \\ \hline
\end{tabular}
\caption{Error metrics on the approximation of $E_2$ using $\vec n =(4,4), \vec k = (2,2)$.}
\label{tab:E2_4}
\end{table}

A comparison between Table~\ref{tab:E1_2} and Table~\ref{tab:E1_4}, as well as Table~\ref{tab:E2_2} and Table~\ref{tab:E2_4}, highlights the performance variations when approximating $E_1$ and $E_2$, respectively.  Notably, the errors reported in Tables~\ref{tab:E1_4} and \ref{tab:E2_4} are lower than their counterparts, demonstrating that the accuracy of the approximant configured with $\vec n=(4,4)$ and $\vec k = (2,2)$ is significantly higher. This behavior aligns with the theoretical framework, as the remainder term in \eqref{eq:approximation} reduces to $\mathcal{O}(z^{-3}w^{-3})$ instead of $\mathcal{O}(z^{-2}w^{-2})$. These experiments indicate that increasing the degree of the denominator polynomial $U_{\vec n,\vec k}^{(\vec\alpha,\vec\beta,\gamma)}$ leads to a substantial improvement in accuracy. Conversely, this enhancement in precision entails a higher computational cost. As the multi-indices scale from $\vec n=(2,2)$ to $\vec n=(4,4)$, the algebraic degree of the denominator polynomial $U_{\vec n,\vec k}^{(\vec\alpha,\vec\beta,\gamma)}$ doubles, inducing a substantial increase in both algorithmic complexity and execution time during the symbolic generation and numerical evaluation of the approximants. Consequently, selecting the optimal approximation configuration requires a careful balance between the targeted numerical accuracy and the available computational resources.

\section{Conclusions}\label{sec:conclusions}

We conclude by summarizing the main contributions of this work. We have defined a family of bivariate polynomials satisfying orthogonality conditions with respect to more than one measure. These polynomials are obtained through the composition of commuting Rodrigues operators, resulting in a Rodrigues formula that
\begin{itemize}
    \item extends the univariate Rodrigues formula for multiple orthogonal polynomials \eqref{eq:rodrigues-1-var} by increasing the number of variables; and
    \item extends the bivariate Rodrigues formula for orthogonal polynomials on the triangle \eqref{eq:Rodrigues-2-var} by increasing the number of measures.
\end{itemize}

Regarding the relation between multiple orthogonality and Hermite--Padé approximation, we have extended the classical univariate framework to the bivariate setting. More precisely, we introduced a Hermite--Padé-type approximation problem for the functions $E_j(z,w)$ defined in \eqref{eq:Ej-integral}, and constructed approximants of the form
\[
\frac{\Phi_{\vec n,\vec k}^{(j)}(z,w)}
     {U_{\vec n,\vec k}^{(\vec\alpha,\vec\beta,\gamma)}(z,w)}.
\]
In this construction, the polynomial
$U_{\vec n,\vec k}^{(\vec\alpha,\vec\beta,\gamma)}$
plays a role analogous to that of the univariate multiple orthogonal polynomial
$P_{\vec n}$.
Unlike the univariate case, however, the numerator
$\Phi_{\vec n,\vec k}^{(j)}$
is not a polynomial, but can still be computed explicitly in terms of polynomials and hypergeometric functions.

A further distinction with respect to the univariate theory concerns uniqueness. In the classical setting, whenever $P_{\vec n}$ exists, it is uniquely determined up to a multiplicative constant. In contrast, the bivariate approximation conditions may admit several polynomial solutions. The polynomial family introduced in \eqref{eq:pol-definition} provides one such solution through a Rodrigues-type construction.

Future work includes extending other families of multivariate orthogonal polynomials given in terms of Rodrigues-type formulas to the multiple orthogonal setting. Furthermore, we aim to 
thoroughly investigate potential extensions of the Hermite--Pad\'{e} 
approximation for multivariate multiple orthogonal polynomials, seeking a 
genuine rational approximation.

\section*{CRediT authorship contribution statement}

All authors contributed equally.

\section*{Acknowledgments}

The work of authors L. Fernández and J. A. Villegas is partially supported by grants PID2023-149117NB-I00, PID2024-155133NB-I00 and CEX
2020-001105-M, all funded by ``Ministerio de Ciencia, Innovación y Universidades'' (MICIU/AEI/10.13039/501100011033 and ERDF/EU), Spain. 

The work of author A. Foulquié-Moreno is partially supported by grant UID/4106/2025 (\url{https://doi.org/\\10.54499\\UID/04106/2025}), funded by ``Fundação para a Ciência e a Tecnologia (FCT)''.

\bibliographystyle{plain}
\bibliography{references}{}

@article{ABVA03,
 ISSN = {00029947},
 URL = {http://www.jstor.org/stable/1194742},
 author = {Aptekarev, A. I. and Branquinho, A. and Van Assche, W.},
 journal = {Transactions of the American Mathematical Society},
 number = {10},
 pages = {3887--3914},
 publisher = {American Mathematical Society},
 title = {Multiple Orthogonal Polynomials for Classical Weights},
 urldate = {2025-10-13},
 volume = {355},
 year = {2003}
}

@article{ACVA01,
    author = {Arves\'u, J. and Coussement, J. and Van Assche, W.},
    year = {2001},
    month = {10},
    pages = {19-45},
    title = {Some Discrete Multiple Orthogonal Polynomials},
    volume = {153},
    journal = {Journal of Computational and Applied Mathematics},
    doi = {10.1016/S0377-0427(02)00597-6}
}

@article{ADL23,
	AUTHOR = {Aptekarev, A. I. and Dyachenko, A. and Lysov, V.},
	TITLE = {On Perfectness of Systems of Weights Satisfying Pearson’s Equation with Nonstandard Parameters},
	JOURNAL = {Axioms},
	VOLUME = {12},
	YEAR = {2023},
	NUMBER = {1},
	ARTICLE-NUMBER = {89},
	URL = {https://www.mdpi.com/2075-1680/12/1/89},
	ISSN = {2075-1680},
	DOI = {10.3390/axioms12010089}
}

@article{AFPP09,
	title = {{A matrix Rodrigues formula for classical orthogonal polynomials in two variables}},
	journal = {Journal of Approximation Theory},
	volume = {157},
	number = {1},
	pages = {32-52},
	year = {2009},
	issn = {0021-9045},
	doi = {https://doi.org/10.1016/j.jat.2008.04.018},
	url = {https://www.sciencedirect.com/science/article/pii/S002190450800155X},
	author = { {\'Alvarez de Morales}, M. and Fern\'andez, L. and P\'erez, T. E. and Pi\~nar, M. A.}
}

@incollection{AFPP08,
    author    = {{\'{A}lvarez de Morales}, M. and Fern\'{a}ndez, L. and P\'{e}rez, T. E. and Pi\~{n}ar, M. A.},
    title     = {A {S}tieltjes function in two variables},
    booktitle = {Approximation Theory {XII}: {S}an {A}ntonio 2007},
    editor    = {Neamtu, M. and Schumaker, L. L.},
    series    = {Mod. Methods Math.},
    pages     = {1--13},
    publisher = {Nashboro Press},
    address   = {Brentwood, TN},
    year      = {2008},
    isbn      = {978-0-9728482-9-9}
}

@book{AK26,
	  title={Fonctions hyperg{\'e}om{\'e}triques et hypersph{\'e}riques: polynomes d'Hermite},
	  author={Appell, P. and De F{\'e}riet, J. K.},
	  year={1926},
	  publisher={Gauthier-villars}
}

@article{Ape79,
     author  = {Ap\'{e}ry, R.},
     title   = {Irrationalit\'{e} de $\zeta(2)$ et $\zeta(3)$},
     journal = {Ast\'{e}risque},
     number  = {61},
     pages   = {11--13},
     year    = {1979},
     note    = {Journ\'{e}es Arithm\'{e}tiques de Luminy},
     url     = {https://www.numdam.org/item/AST_1979__61__11_0/}
}

@article{BCVA05,
    title = {Multiple {Wilson} and {Jacobi--Piñeiro} polynomials},
    journal = {Journal of Approximation Theory},
    volume = {132},
    number = {2},
    pages = {155-181},
    year = {2005},
    issn = {0021-9045},
    doi = {https://doi.org/10.1016/j.jat.2004.12.001},
    url = {https://www.sciencedirect.com/science/article/pii/S0021904504002084},
    author = {Beckermann, B. and Coussement, J. and Van Assche, W.}
}

@Article{BDFM25,
author={Branquinho, A.
and D{\'i}az, J. E. F.
and {Foulqui{\'e}-Moreno}, A.
and Ma{\~{n}}as, M.},
title={{Uniform multiple orthogonal polynomials and Markov chains}},
journal={Numerical Algorithms},
year={2025},
month={Aug},
day={05},
issn={1572-9265},
doi={10.1007/s11075-025-02195-6},
url={https://doi.org/10.1007/s11075-025-02195-6}
}

@Article{BDFMAF23,
author={Branquinho, A.
and D{\'i}az, J. E. F.
and {Foulqui{\'e}-Moreno}, A.
and Ma{\~{n}}as, M.
and {\'A}lvarez-Fern{\'a}ndez, C.},
title={{Jacobi--Pi{\~{n}}eiro Markov chains}},
journal={Revista de la Real Academia de Ciencias Exactas, F{\'i}sicas y Naturales. Serie A. Matem{\'a}ticas},
year={2023},
month={Oct},
day={25},
volume={118},
number={1},
pages={15},
issn={1579-1505},
doi={10.1007/s13398-023-01510-x},
url={https://doi.org/10.1007/s13398-023-01510-x}
}

@article {BK04,
    AUTHOR = {Bleher, P. M. and Kuijlaars, A. B. J.},
     TITLE = {Random matrices with external source and multiple orthogonal polynomials},
   JOURNAL = {Int. Math. Res. Not.},
  FJOURNAL = {International Mathematics Research Notices},
      YEAR = {2004},
    NUMBER = {3},
     PAGES = {109--129},
      ISSN = {1073-7928,1687-0247},
   MRCLASS = {82B41 (28C10 33C90 60B15 82B44)},
  MRNUMBER = {2038771},
MRREVIEWER = {Dimitri\ Petritis},
       DOI = {10.1155/S1073792804132194},
       URL = {https://doi.org/10.1155/S1073792804132194},
}

@article{BK05,
	     author = {Bleher, P. M. and Kuijlaars, A. B. J.},
	     title = {Integral representations for multiple {Hermite} and multiple {Laguerre} polynomials},
	     journal = {Annales de l'Institut Fourier},
	     pages = {2001--2014},
	     year = {2005},
	     publisher = {Association des Annales de l{\textquoteright}institut Fourier},
	     volume = {55},
	     number = {6},
	     doi = {10.5802/aif.2148},
	     zbl = {1084.33008},
	     mrnumber = {2187942},
	     language = {en},
	     url = {https://aif.centre-mersenne.org/articles/10.5802/aif.2148/}
}

@article{BR01,
	author = {Ball, K. and Rivoal, T.},
	year = {2001},
	month = {10},
	pages = {193-207},
	title = {Irrationalité d'une infinité de valeurs de la fonction z\^eta aux entiers impairs},
	volume = {146},
	journal = {Inventiones Mathematicae},
	doi = {10.1007/s002220100168}
}

@article{CB00,
	  title     = {{Multivariate orthogonal polynomials, homogeneous Pad{\'e} approximants and Gaussian cubature}},
	  author    = {Benouahmane, B. and Cuyt, A. A. M.},
	  journal   = {Numerical Algorithms},
	  volume    = {24},
	  number    = {1},
	  pages     = {1--15},
	  year      = {2000},
	  publisher = {Springer}
}

@article{CDO15,
    title = {{Multiple orthogonal polynomials on the unit circle. Normality and recurrence relations}},
    journal = {Journal of Computational and Applied Mathematics},
    volume = {284},
    pages = {115-132},
    year = {2015},
    note = {OrthoQuad 2014},
    issn = {0377-0427},
    doi = {https://doi.org/10.1016/j.cam.2014.11.004},
    url = {https://www.sciencedirect.com/science/article/pii/S0377042714004750},
    author = {{Cruz-Barroso}, R. and {D\'iaz-Mendoza}, C. and Orive, R.}
}

@book{Chi78,
	author = {Chihara, T. S.},
	address = {New York},
	booktitle = {An introduction to orthogonal polynomials},
	isbn = {0677041500},
	language = {eng},
	lccn = {76041598},
	publisher = {Gordon and Breach},
	series = {Mathematics and its applications (13)},
	title = {An introduction to orthogonal polynomials },
	url = {http://bvbm2.bib-bvb.de:8993/F?func=service&doc_library=BVB01&doc_number=001480531&line_number=0001&func_code=DB_RECORDS&service_type=MEDIA},
	year = {1978},
}

@article{CLY16,
	  title   = {{On tensor decomposition, sparse interpolation and Pad{\'e} approximation}},
	  author  = {Cuyt, A. A. M. and Lee, W. and Yang, X.},
	  journal = {Ja{\'e}n Journal of Approximation},
	  volume  = {8},
	  number  = {1},
	  pages   = {33--58},
	  year    = {2016}
}

@article{Cuy83,
	  title   = {{Multivariate Pad\'e-approximants}},
	  journal = {Journal of Mathematical Analysis and Applications},
	  volume  = {96},
	  number  = {1},
	  pages   = {283-293},
	  year    = {1983},
	  issn    = {0022-247X},
	  doi     = {https://doi.org/10.1016/0022-247X(83)90041-0},
	  url     = {https://www.sciencedirect.com/science/article/pii/0022247X83900410},
	  author  = {Cuyt, A. A. M.}
}

@article{Cuy86,
	  author  = {Cuyt, A. A. M.},
	  title   = {{Multivariate Pad\'e approximants revisited}},
	  journal = {BIT Numerical Mathematics},
	  year    = {1986},
	  volume  = {26},
	  issue   = {1},
	  pages   = {71-79},
	  doi     = {10.1007/bf01939363}
}

@article{Cuy99,
	  title   = {{How well can the concept of Pad\'e approximant be generalized to the multivariate case?}},
	  journal = {Journal of Computational and Applied Mathematics},
	  volume  = {105},
	  number  = {1},
	  pages   = {25-50},
	  year    = {1999},
	  issn    = {0377-0427},
	  doi     = {https://doi.org/10.1016/S0377-0427(99)00028-X},
	  url     = {https://www.sciencedirect.com/science/article/pii/S037704279900028X},
	  author  = {Cuyt,  A. A. M.}
}

@article{DK07,
	title = {Multiple orthogonal polynomials of mixed type and non-intersecting Brownian motions},
	journal = {Journal of Approximation Theory},
	volume = {146},
	number = {1},
	pages = {91-114},
	year = {2007},
	issn = {0021-9045},
	doi = {https://doi.org/10.1016/j.jat.2006.12.001},
	url = {https://www.sciencedirect.com/science/article/pii/S0021904506002036},
	author = {Daems, E. and Kuijlaars, A. B. J. }
}

@misc{DS16,
      title={{Monotonicity of Zeros of Jacobi-Angelesco polynomials}}, 
      author={{dos Santos}, E. J. C.},
      howpublished = {arXiv preprint arXiv:1603.03762 [math.CA]},
      year={2016},
      eprint={1603.03762},
      archivePrefix={arXiv},
      primaryClass={math.CA},
      url={https://arxiv.org/abs/1603.03762}, 
}

@book{DX14,
	  place      = {Cambridge},
	  edition    = {2},
	  series     = {Encyclopedia of Mathematics and its Applications},
	  title      = {{Orthogonal Polynomials of Several Variables}},
	  doi        = {10.1017/CBO9781107786134},
	  publisher  = {Cambridge University Press},
	  author     = {Dunkl, C. F. and Xu, Y.},
	  year       = {2014},
	  collection = {Encyclopedia of Mathematics and its Applications}
}

@article{Fis04,
     author  = {Fischler, S.},
     title   = {Irrationalit\'{e} de valeurs de $\zeta$},
     journal = {Ast\'{e}risque},
     number  = {294},
     pages   = {27--62},
     year    = {2004},
     note    = {S\'{e}minaire Bourbaki, Exp. nº 910},
     url     = {https://www.numdam.org/item/SB_2002-2003__45__27_0/}
}

@article{FV26,
	  title   = {Multiple orthogonal polynomials of two real variables},
	  journal = {Journal of Mathematical Analysis and Applications},
	  volume  = {553},
	  number  = {1},
	  pages   = {129811},
	  year    = {2026},
	  issn    = {0022-247X},
	  doi     = {https://doi.org/10.1016/j.jmaa.2025.129811},
	  url     = {https://www.sciencedirect.com/science/article/pii/S0022247X2500592X},
	  author  = {Fern\'andez, L. and Villegas, J. A.},
}

@article{GI21,
  title={An urn model for the {J}acobi-{P}i\~neiro polynomials},
  author={Grunbaum, F. A. and {Domínguez de la Iglesia}, M.},
  year={2021},
  url={https://api.semanticscholar.org/CorpusID:233481717},
  journal={Proceedings of the American Mathematical Society}
}

@article{Herm1865, 
	title = {{Sur quelques developpments en s\'erie de fonctions ´
	de plusieurs variables}},
	journal = {Comptes Rendus de l’Academi\'e des Sciences},
	author={Hermite, C.}, 
	year={1865}, 
	pages={370–377} 
}

@book{Ism05,
	  author    = {Ismail, M. E. H.},
	  address   = {Cambridge (England)},
	  booktitle = {Classical and quantum orthogonal polynomials in one variable},
	  isbn      = {0521782015},
	  language  = {eng},
	  publisher = {Cambridge University Press},
	  series    = {Encyclopedia of mathematics and its applications},
	  volume    = {98},
	  title     = {Classical and quantum orthogonal polynomials in one variable},
	  year      = {2005}
}

@incollection{Koo75,
    author    = {Koornwinder, T. H.},
    title     = {Two-variable analogues of the classical orthogonal polynomials},
    editor    = {Askey, Richard A.},
    booktitle = {Theory and Application of Special Functions},
    publisher = {Academic Press},
    address   = {New York, NY},
    pages     = {435--495},
    year      = {1975},
    isbn      = {978-0-12-064850-4},
    doi       = {10.1016/B978-0-12-064850-4.50015-X},
    url       = {https://www.sciencedirect.com/science/article/pii/B978012064850450015X}
}

@article{KV24,
  title={Angelesco and AT systems on the Unit Circle},
  author={Kozhan, R. and Vaktn{\"a}s, M.},
  journal={arXiv preprint arXiv:2410.12094},
  year={2024}
}

@article{KV26,
  title={Zeros of Laurent multiple orthogonal polynomials on the unit circle},
  author={Kozhan, R. and Vaktn{\"a}s, M.},
  journal={arXiv preprint arXiv:2603.21468},
  year={2026}
}

@Article{Lau1893,
author={Lauricella, G.},
title={Sulle funzioni ipergeometriche a piu variabili},
journal={Rendiconti del Circolo Matematico di Palermo},
year={1893},
month={Dec},
day={01},
volume={7},
number={1},
pages={111-158},
issn={1973-4409},
doi={10.1007/BF03012437},
url={https://doi.org/10.1007/BF03012437}
}

@article{MFOSL22,
  title     = {{Electrostatic Partners and Zeros of Orthogonal and Multiple Orthogonal Polynomials}},
  volume    = {58},
  issn      = {1432-0940},
  url       = {http://dx.doi.org/10.1007/s00365-022-09609-x},
  doi       = {10.1007/s00365-022-09609-x},
  number    = {2},
  journal   = {Constructive Approximation},
  publisher = {Springer Science and Business Media LLC},
  author    = {{Mart\'inez-Finkelshtein}, A. and Orive, R. and {S\'anchez-Lara}, J.},
  year      = {2022},
  month     = dec,
  pages     = {271--342}
}

@article{MFVA16,
	  author  = {{Mart\'inez-Finkelshtein}, A. and Van Assche, W.},
	  journal = {Notices of the American Mathematical Society},
	  month   = {10},
	  number  = {09},
	  pages   = {1029--1031},
	  title   = {{WHAT IS...A multiple orthogonal polynomial?}},
	  volume  = {63},
	  year    = {2016},
	  doi     = {10.1090/noti1430},
	  url     = {https://doi.org/10.1090/noti1430}
}

@article{MRW25,
  title={Bivariate multiple orthogonal polynomials of mixed type on the step-line},
  author={Ma\~nas, M. and Rojas, M. and Wu, J.},
  journal={arXiv preprint arXiv:2507.17694},
  year={2025}
}

@book{NS91,
  title     = {Rational approximations and orthogonality},
  author    = {Nikishin, E. M. and Sorokin, V. N.},
  volume    = {92},
  year      = {1991},
  publisher = {American Mathematical Society Providence, RI}
}

@book{OLBC10,
  author = {Olver, F. and Lozier, D. and Boisvert, R. and Clark, C.},
  title = {The NIST Handbook of Mathematical Functions},
  year = {2010},
  month = {2010-05-12 00:05:00},
  publisher = {Cambridge University Press, New York, NY},
  language = {en},
}

@article {Pin87,
	    AUTHOR = {Pi\~neiro, L. R.},
	     TITLE = {On simultaneous approximations for some collection of {M}arkov functions},
	   JOURNAL = {Vestnik Moskov. Univ. Ser. I Mat. Mekh.},
	  FJOURNAL = {Vestnik Moskovskogo Universiteta. Seriya I. Matematika, Mekhanika},
	      YEAR = {1987},
	    NUMBER = {2},
	     PAGES = {67--70, 103},
	      ISSN = {0579-9368},
	   MRCLASS = {41A21},
  MRNUMBER = {884516},
}

@Article{Sor02,
	  author={Sorokin, V. N.},
	  title={{The Hermite--Pad\'e Approximations of Generalized Hypergeometric Series in Two Variables}},
	  journal={Siberian Mathematical Journal},
	  year={2002},
	  month={Jul},
	  day={01},
	  volume={43},
	  number={4},
	  pages={719-730},
	  issn={1573-9260},
	  doi={10.1023/A:1016336605594},
	  url={https://doi.org/10.1023/A:1016336605594}
}

@book{Sue99,
	author = {Suetin, P. K.},
	month = {8},
	publisher = {CRC Press},
	title = {{Orthogonal polynomials in two variables}},
	year = {1999},
}

@book{Szego75,
	  title={Orthogonal Polynomials},
	  author={Szeg{\H{o}}, G.},
	  isbn={9780821810231},
	  lccn={77476087},
	  series={American Math. Soc: Colloquium publ},
	  url={https://books.google.es/books?id=2phzwgEACAAJ},
	  year={1975},
	  publisher={American Mathematical Society}
}

@article{VA06,
	author = {Van Assche, W.},
	year = {2006},
	month = {01},
	pages = {},
	title = {Padé and Hermite-Padé Approximation and Orthogonality},
	volume = {2},
	journal = {Surveys in Approximation Theory}
}

@inbook{VA20,
  title      = {{Orthogonal and Multiple Orthogonal Polynomials, Random Matrices, and Painlev\'e Equations}},
  isbn       = {9783030367442},
  author     = {Van Assche, W.},
  booktitle  = {Orthogonal Polynomials},
  publisher  = {Springer International Publishing},
  year       = {2020},
  pages      = {629--683},
  chapter    = {13 (Part II)}
}

@article{Zud04,
     author = {Zudilin, W.},
     title = {Arithmetic of linear forms involving odd zeta values},
     journal = {Journal de th\'eorie des nombres de Bordeaux},
     pages = {251--291},
     publisher = {Universit\'e Bordeaux 1},
     volume = {16},
     number = {1},
     year = {2004},
     doi = {10.5802/jtnb.447},
     zbl = {02184645},
     mrnumber = {2145585},
     language = {en},
     url = {https://www.numdam.org/articles/10.5802/jtnb.447/}
}

\end{document}